\documentclass[12pt]{amsart}

\usepackage{amsmath,amssymb,amsbsy,amsfonts,amsthm,latexsym,
                        amsopn,amstext,amsxtra,euscript,amscd,mathrsfs,color,bm,cite}

\usepackage{float}
\usepackage[english]{babel}
\usepackage{mathtools}
\usepackage{todonotes}
\usepackage{url}
\usepackage[colorlinks,linkcolor=blue,anchorcolor=blue,citecolor=blue,backref=page]{hyperref}
\RequirePackage{mathrsfs} \let\mathcal\mathscr

\usepackage{enumerate}

\allowdisplaybreaks

\def\le{\leqslant}
\def\ge{\geqslant}

\def\leq{\leqslant}

\usepackage{mathtools}
\usepackage{todonotes}
\usepackage[norefs,nocites]{refcheck}

\usepackage{color}

\usepackage[english]{babel}
\usepackage{mathtools}
\usepackage{todonotes}
\usepackage{url}
\usepackage[colorlinks,linkcolor=blue,anchorcolor=blue,citecolor=blue,backref=page]{hyperref}

\begin{document}

\newtheorem{theorem}{Theorem}
\newtheorem{lemma}[theorem]{Lemma}
\newtheorem{proposition}[theorem]{Proposition}
\newtheorem{claim}[theorem]{Claim}
\newtheorem{cor}[theorem]{Corollary}
\newtheorem{prop}[theorem]{Proposition}
\newtheorem{rem}[theorem]{Remark}
\newtheorem{definition}[theorem]{Definition}
\newtheorem{question}[theorem]{Question}
\newtheorem{conj}[theorem]{Conjecture}
\newcommand{\hh}{{{\mathrm h}}}

\numberwithin{equation}{section}
\numberwithin{theorem}{section}
\numberwithin{table}{section}

\numberwithin{figure}{section}

\def\sssum{\mathop{\sum\!\sum\!\sum}}
\def\ssum{\mathop{\sum\ldots \sum}}
\def\iint{\mathop{\int\ldots \int}}

%\def\squareforqed{\hbox{\rlap{$\sqcap$}$\sqcup$}}
%\def\qed{\ifmmode\squareforqed\else{\unskip\nobreak\hfil
%\penalty50\hskip1emrll\nobreak\hfil\squareforqed
%\parfillskip=0pt\finalhyphendemerits=0\endgraf}\fi}%%

%  use the AMS-Euler Fraktur fonts
%%%%%%%%%%%%%%%%%%%%%%%%%%%%%%%%%%
\newfont{\teneufm}{eufm10}
\newfont{\seveneufm}{eufm7}
\newfont{\fiveeufm}{eufm5}
%%%%%%%%%%%%%%%%%%%%%%%%%%%%%%%%%
%
%  allow automatic size selection in math mode
%
%%%%%%%%%%%%%%%%%%%%%%%%%%%%%%%%%
\newfam\eufmfam
     \textfont\eufmfam=\teneufm
\scriptfont\eufmfam=\seveneufm
     \scriptscriptfont\eufmfam=\fiveeufm
%%%%%%%%%%%%%%%%%%%%%%%%%%%%%%%%%
%
%  \frak works on a single symbol at a time...
%
\def\frak#1{{\fam\eufmfam\relax#1}}

\newcommand{\bflambda}{{\boldsymbol{\lambda}}}
\newcommand{\bfmu}{{\boldsymbol{\mu}}}
\newcommand{\bfxi}{{\boldsymbol{\xi}}}
\newcommand{\bfrho}{{\boldsymbol{\rho}}}

\def\fA{{\mathfrak A}}
\def\fB{{\mathfrak B}}
\def\fC{{\mathfrak C}}
\def\fI{{\mathfrak I}}
\def\fJ{{\mathfrak J}}
\def\fK{{\mathfrak K}}
\def\fm{{\mathfrak m}}
\def\fM{{\mathfrak M}}
\def\fS{{\mathfrak S}}
 \def\fW{{\mathfrak W}}
  \def\fU{{\mathfrak U}}

\def \balpha{\bm{\alpha}}
\def \bbeta{\bm{\beta}}
\def \bgamma{\bm{\gamma}}
\def \blambda{\bm{\lambda}}
\def \bchi{\bm{\chi}}
\def \bphi{\bm{\varphi}}
\def \bpsi{\bm{\psi}}
\def \bomega{\bm{\omega}}
\def \btheta{\bm{\vartheta}}

\def \bzeta{\bm{\zeta}}
\def \bxi{\bm{\xi}}

\def\eqref#1{(\ref{#1})}

\def\vec#1{\mathbf{#1}}

%\def\squareforqed{\hbox{\rlap{$\sqcap$}$\sqcup$}}
%\def\qed{\ifmmode\squareforqed\else{\unskip\nobreak\hfil
%\penalty50\hskip1emrll\nobreak\hfil\squareforqed
%\parfillskip=0pt\finalhyphendemerits=0\endgraf}\fi}

%%%%%%%%%%%%%%%%%%%%%%%%%
% Alphabet calligraphie %
%%%%%%%%%%%%%%%%%%%%%%%%%
\def\cA{{\mathcal A}}
\def\cB{{\mathcal B}}
\def\cC{{\mathcal C}}
\def\cD{{\mathcal D}}
\def\cE{{\mathcal E}}
\def\cF{{\mathcal F}}
\def\cG{{\mathcal G}}
\def\cH{{\mathcal H}}
\def\cI{{\mathcal I}}
\def\cJ{{\mathcal J}}
\def\cK{{\mathcal K}}
\def\cL{{\mathcal L}}
\def\cM{{\mathcal M}}
\def\cN{{\mathcal N}}
\def\cO{{\mathcal O}}
\def\cP{{\mathcal P}}
\def\cQ{{\mathcal Q}}
\def\cR{{\mathcal R}}
\def\cS{{\mathcal S}}
\def\cT{{\mathcal T}}
\def\cU{{\mathcal U}}
\def\cV{{\mathcal V}}
\def\cW{{\mathcal W}}
\def\cX{{\mathcal X}}
\def\cY{{\mathcal Y}}
\def\cZ{{\mathcal Z}}
\newcommand{\rmod}[1]{\: (\mbox{mod} \: #1)}
\renewcommand{\bmod}[1]{\,(\mathrm{mod}{\: #1})}

\def\cg{{\mathcal g}}

\def\vr{\mathbf r}

\def\e{{\mathbf{\,e}}}
\def\ep{{\mathbf{\,e}}_p}
\def\eq{{\mathbf{\,e}}_q}

\def\Tr{{\mathrm{Tr}}}
\def\Nm{{\mathrm{Nm}}}

 \def\SS{{\mathbf{S}}}

\def\lcm{{\mathrm{lcm}}}
\def\GL{{\mathrm{GL}}}
\def\ud{{\mathrm{d}}}

\def\({\left(}
\def\){\right)}
\def\fl#1{\left\lfloor#1\right\rfloor}
\def\rf#1{\left\lceil#1\right\rceil}

\def\mand{\qquad \mbox{and} \qquad}

\let\ve=\varepsilon

\newcommand{\commB}[2][]{\todo[#1,color=blue!60]{B: #2}}
\newcommand{\commI}[2][]{\todo[#1,color=green!60]{I: #2}}
\newcommand{\commII}[2][]{\todo[#1,color=magenta!60]{I: #2}}
\newcommand{\commW}[2][]{\todo[#1,color=yellow!60]{W: #2}}
\newcommand{\commX}[2][]{\todo[#1,color=red!60]{X: #2}}

%%%%%%%%%%%%%%%%%%%%%%%%%%%%%%%%%%%%%%%%%%%%%%%%%%%%%%%%
%%%%%%%%%%%%%%%%%%%%%%%%%%%%%%%%%%%%%%%%%%%%%%%%%%%%%%%%
%%%%%%%%%%%%%%%%%%%%%%%%%%%%%%%%%%%%%%%%%%%%%%%%%%%%%%%%
%%%%%%%%%%%%%%%%%%%%%%%%%%%%%%%%%%%%%%%%%%%%%%%%%%%%%%%%

%%%%%%%  END OF STANDARD STUFF %%%%%%%%%

%%%%%%%%%%%%%%%%%%%%%%%%%%%%%%%%%%%%%%%%%%%%%%%%%%%%%%%%
%%%%%%%%%%%%%%%%%%%%%%%%%%%%%%%%%%%%%%%%%%%%%%%%%%%%%%%%
%%%%%%%%%%%%%%%%%%%%%%%%%%%%%%%%%%%%%%%%%%%%%%%%%%%%%%%%
%%%%%%%%%%%%%%%%%%%%%%%%%%%%%%%%%%%%%%%%%%%%%%%%%%%%%%%
%%%%%%%%%%%
%%% Spell

\hyphenation{re-pub-lished}

\mathsurround=1pt

\def\bfdefault{b}

\def \F{{\mathbb F}}
\def \K{{\mathbb K}}
\def \N{{\mathbb N}}
\def \Z{{\mathbb Z}}
\def \Q{{\mathbb Q}}
\def \R{{\mathbb R}}
\def \C{{\mathbb C}}
\def\Fp{\F_p}
\def \fp{\Fp^*}

\def\Kmnp{\cK_p(m,n)}
\def\Kmnq{\cK_q(m,n)}
\def\Kamnq{\cK_q(m,an)}
\def\KKap{\cH_p(a)}
\def\KKaq{\cH_q(a)}
\def\KKmnp{\cH_p(m,n)}
\def\KKmnq{\cH_q(m,n)}

\def\Kl{{\mathsf K}}

\def\Klmnp{\cK_p(\ell, m,n)}
\def\Klmnq{\cK_q(\ell, m,n)}

\def \SALMNq {\cS_q(\balpha;\cL,\cI,\cJ)}
\def \SALMNp {\cS_p(\balpha;\cL,\cI,\cJ)}

\def \SACXMQX {\fS(\balpha,\bzeta, \bxi; M,Q,X)}

\def \balpha{\bm{\alpha}}
\def \bbeta{\bm{\beta}}
\def \bgamma{\bm{\gamma}}
\def \blambda{\bm{\lambda}}
\def \bchi{\bm{\chi}}
\def \bphi{\bm{\varphi}}
\def \bpsi{\bm{\psi}}

\def\SXaq{S(X;a,q)}
\def\MXaq{\mathrm{M}(X;a,q)}
\def\RXaq{\mathrm{R}(X;a,q)}
\def\RXap{\mathrm{R}(X;a,p)}
\def\EAR{\cE(X;\cA,q)}
\def\EaR{\cE^*(X;\cA,q)}
\def\eps{\varepsilon}

\def\SIJq{S_q(\balpha; \cI,\cJ)}
\def\SIJp{S_p(\balpha; \cI,\cJ)}
\def\Sq{S_{q}(\balpha)}
\def\Sqa{S_{q,a}(\balpha)}
\def\Spa{S_{p,a}(\balpha)}
\def\Sq{S_{q}(\balpha)}
\def\Sp{S_{p}(\balpha)}
\def\SMp{S^\sharp_p(\balpha)}

\def\UMNp{U_p(\balpha, \bbeta; M,N)}
\def\VMNp{V_p(\balpha, \bbeta; M,N)}
\def\LG#1{\(\frac{#1}{p}\)}

\def\SMJq{S_q(\balpha; \cM,\cJ)}
\def\SMJp{S_p(\balpha; \cM,\cJ)}

\def\WIKq{W_{q,a}(\bgamma; \cI,K)}
\def\WIKp{W_{p,a}(\bgamma; \cI,K)}
\def\WMKq{W_{q,a}(\bgamma; \cM,K)}
\def\WMKp{W_{p,a}(\bgamma; \cM,K)}

\def\WaIKq{W_{q,a}(\balpha, \bgamma)}
\def\WaIKp{W_{p,a}(\balpha, \bgamma)}
\def\WacIKqr{W_{q,a}(\balpha, \bgamma;r,c)}
\def\WacIKp{W_{p,a,c}(\balpha, \bgamma)}
\def\WaMKq{W^\sharp_{q,a}(\balpha, \bgamma)}
\def\WaMKp{W^\sharp_{p,a}(\balpha, \bgamma)}

\def\WAMNQ{W_q(\balpha, \bgamma; M,N,Q)}

\def\RIJp{\cR_p(\cI,\cJ)}
\def\RIJq{\cR_q(\cI,\cJ)}

\def\TWXJp{\cT_p(\bomega;\cX,\cJ)}
\def\TWXJq{\cT_q(\bomega;\cX,\cJ)}
\def\TWpXJp{\cT_p(\bomega_p;\cX,\cJ)}
\def\TWqXJq{\cT_q(\bomega_q;\cX,\cJ)}
\def\TWJq{\cT_q(\bomega;\cJ)}
\def\TWqJq{\cT_q(\bomega_q;\cJ)}

\newcommand{\supp}{\operatorname{supp}}

  \def \kbar{k^{-1} }
 \def \xbar{x^{-1} }
  \def \ybar{y^{-1} }

\title[Bilinear Forms with Kloosterman Sums]{Bounds on Bilinear Forms with Kloosterman Sums}

 \author[B. Kerr] {Bryce Kerr}
\address{BK: School of Science, University of New South Wales, Canberra, ACT, 2612, Australia}
\email{bryce.kerr@unsw.edu.au}

 \author[I. E. Shparlinski] {Igor E. Shparlinski}

\address{IES: Department of Pure Mathematics, University of New South Wales,
Sydney, NSW 2052, Australia}
\email{igor.shparlinski@unsw.edu.au}

 \author[X. Wu]{Xiaosheng Wu}
\address{XW: School of Mathematics, Hefei University of Technology, Hefei 230009, P.R. China}
\email{xswu@amss.ac.cn}

 \author[P. Xi]{Ping Xi}
 \address{PX: School of Mathematics and Statistics, Xi'an Jiaotong University, Xi'an 710049,
P.R. China}
\email{ping.xi@xjtu.edu.cn}

\begin{abstract}
We prove new bounds on bilinear forms with Kloosterman sums, complementing and improving a series of results by \'E.~Fouvry, E.~Kowalski and Ph.~Michel (2014), V.~Blomer,  \'E.~Fouvry, E.~Kowalski, Ph.~Michel and D. Mili{\'c}evi{\'c} (2017), E.~Kowalski, Ph.~Michel and W.~Sawin (2019, 2020)
and I.~E.~Shparlinski (2019). These improvements rely on new estimates for Type~II bilinear forms with incomplete Kloosterman sums. We also establish new estimates for bilinear forms with one variable from an arbitrary set by introducing techniques from additive combinatorics over prime fields. Some of these bounds have found a crucial application in the recent work of Wu (2020) on asymptotic formulas for the fourth moments of Dirichlet $L$-functions.
As new applications, an estimate for higher moments of averages of Kloosterman sums and the distribution of divisor function in a family of arithmetic progressions are also given.
\end{abstract}


\maketitle

\tableofcontents

\section{Introduction}\label{sec:introduction}
\subsection{Backgrounds}
This paper concerns bilinear forms of the shape
$$
\sum_{m \in \cI} \sum_{n \in \cJ} \alpha_m \beta_n\phi(m,n)
$$
with $\phi$ coming from complete or incomplete Kloosterman sums, where $\cI,\cJ$ are two intervals given by
 \begin{equation}\label{eq:Interval-IJ}
\cI = \{M_0+1, \ldots, M_0+M\},\ \cJ = \{N_0+1, \ldots, N_0+N\} \subseteq \Z
\end{equation}
for $M,N\geqslant1$, and $\balpha = \{\alpha_m\}$, $\bbeta = \{\beta_n\}$
are arbitrary complex coefficients supported on $\cI$ and $\cJ$, respectively.
%\commI{I prefer $M_0, N_0$ to $M_1,N_1$, as $M_1,N_1$ get different meaning in sec.9.
%This led to some adjustments in sec.5.2}
According to that $\bbeta$ is the indicator function of $\cJ$ or not, we call the above bilinear forms \emph{Type~I} or \emph{Type~II} sums respectively. %%commB added respectively

As usual we define the \emph{Kloosterman sum}
$$
\Kmnq = \sum_{x \in \Z_q^*} \eq(m x +n \xbar)
$$
for $q\in\Z^+$ and $m,n\in\Z$,
where  $\Z_q^*$ denotes the group of units in the residue ring  $\Z_q$  modulo $q$,
$\xbar$ is the multiplicative inverse of $x$ in  $\Z_q^*$ and
$\eq(z) = \exp(2 \pi i z/q)$. Put
$$
S_{q,a}(\balpha,\bbeta)=\sum_{m \in \cI} \sum_{n \in \cJ} \alpha_m \beta_n\Kamnq,
$$
where $a \in \Z$,
and we also write $\Sqa$ for abbreviation if $\bbeta$ is taken as the indicator function of $\cJ$. Moreover, we write $S_q(\balpha,\bbeta)$ and $\Sq$, respectively, in the case $a=1$.

The celebrated Weil bound (see~\cite[Corollary~11.12]{IK04}) gives
\begin{equation}\label{eq:Weil}
\Kmnq\ll  \gcd(m,n,q)^{1/2} q^{1/2+o(1)},
\end{equation}
from which it follows, in the particular case $|\alpha_m|,|\beta_n|\leqslant 1$ for instance, that
\begin{equation}\label{eq:trivial S}
S_{q,a}(\balpha,\bbeta) \ll MNq^{1/2+o(1)}
\end{equation}
uniformly over $a$, see Section~\ref{sec:Not} for the meaning of $A\ll B$. 
We refer this as the \emph{trivial} bound and our focus is
%%commB the goal is to -> and our focus
 to go beyond this boundary as a fundamental problem of independent interest.

On the other hand, the study of $S_{q,a}(\balpha, \bbeta)$, as well as diverse variations, is highly motivated by applications to analytic number theory, including, for example:
\begin{itemize}
\item Asymptotic formulas for moments of $L$-functions in families; see for instance a series of
papers  by
Blomer, Fouvry,  Kowalski, Michel, Mili{\'c}evi{\'c} and Sawin~\cite{BFKMM17a,BFKMM17b, FKM14, KMS17}, %%  and by
Shparlinski~\cite{Sh19} and Wu~\cite{Wu22}. In particular, some bounds of this work are used in the work of  Wu~\cite{Wu22},  which gives the currently best bound on the error term
in such formulas, see Section~\ref{sec: 4th Moment} for more details.
\item Results on the equidistribution of divisor functions in arithmetic progressions,  see Kerr and Shparlinski~\cite{KeS20}, Liu, Shparlinski and Zhang~\cite{LSZ18a}, Wu and Xi~\cite{WX21} and Xi~\cite{Xi18} for instance, see Section~\ref{sec:DIv-AP}
for new applications.
\item Sums of Kloosterman sums with arithmetic weights such as the von
Mangoldt function $\Lambda(n)$, M\"obius function $\mu(n)$ and
divisor function $\tau(n)$,  that is,
$$
 \sum_{n\leq N} w(n) \cK_q(n,a),
$$
 where $w=\Lambda,\mu,\tau$. See, for example, Fouvry,  Kowalski and Michel~\cite{FKM14}, Korolev and Shparlinski~\cite{KoS20}, Kowalski, Michel and Sawin~\cite{KMS20} and Liu, Shparlinski and Zhang~\cite{LSZ19}.
\end{itemize}

%We also note that finite field analogues of the Type-I sums $\SIJq$ have been estimated in~\cite{MacShp}.

A large amount of investigations into $S_{q,a}(\balpha, \bbeta)$ are deeply influenced by the work of Deligne~\cite{De80}
on the Riemann Hypothesis for algebraic varieties over finite fields, as well as the subsequent developments thanks to Laumon, Katz, et al. In particular, Kowalski, Michel and Sawin~\cite{KMS17,KMS20} introduced the  ``shift by $ab$" trick of Vinogradov, Karatsuba and Friedlander--Iwaniec, and transformed $S_{q,a}(\balpha, \bbeta)$ with prime $q$, as well as its extensions, to a certain sum of products of Kloosterman sums with suitable shifts.
The task then reduces to proving that the resultant function comes from some $\ell$-adic sheaves satisfying reasonable purity and irreducibility conditions.
%%commB rephrased end sentences ...and so on.
 The novelty in~\cite{KMS17,KMS20} allows one to go beyond the so-called
\emph{P{\'o}lya--Vinogradov  range} when $M,N = q^{1/2+o(1)}$, which turns out to be very crucial in many applications (see~\cite{KMS17,BFKMMS23} on the analytic theory of $\GL_2$ $L$-functions).

Note that the use of Deligne's work forces one to deal with prime moduli $q$, or squarefree moduli with some extra effort. 
The input from $\ell$-adic cohomology also applies to bilinear forms with a very wide class of kernel functions $\phi$, which are not necessarily Kloosterman sums. This generality admits the diversity of applications of such bilinear forms.
On the other hand, Shparlinski~\cite{Sh19} developed an alternative device which suits very well for bilinear forms with complete and incomplete Kloosterman sums. The argument therein utilizes the exact shape of Kloosterman sums, %comm `and is'-> 'is'
is completely elementary and thus works for arbitrary composite moduli.

%%commB added some more things to the sentence just below

In this paper we enhance the argument of~Shparlinski~\cite{Sh19} by developing elementary point counting methods to obtain effective control on the counting function $J_q(a,K)$ defined by~\eqref{eq:Jq(a,K)}.
We also introduce input from additive combinatorics which allows new estimates for bilinear forms where one variable may come from an arbitrary set. A consequence of such estimates is a new bound for higher moments of Kloosterman sums as given in Theorem~\ref{thm:moments} below. However these results only apply to prime moduli.

In the background of our results is a series of new bounds on Type~II bilinear forms with incomplete Kloosterman sums, see  Section~\ref{sec:TypeII-incomplete-q}. Such sums are of
independent interest as they appear in a number of applications which include
distributions, in earlier approaches by Friedlander and Iwaniec~\cite{FI85} and Heath-Brown~\cite{HB86}, of the ternary divisor function
$$
\tau_3(n)=\#\{(n_1,n_2,n_3)\in\N^3:n_1n_2n_3=n\}
$$
in arithmetic progressions, and a new level of distribution could be guaranteed by some more stronger estimates for bilinear forms with incomplete Kloosterman sums.

One of the most important {\it applications} of our results can be found in the recent work of
Wu~\cite{Wu22}, which gives  new asymptotic formulas on the fourth moments of Dirichlet $L$-functions, improving and generalizing those of Young~\cite{Yo11} and then by Blomer, Fouvry,  Kowalski, Michel and Mili{\'c}evi{\'c}~\cite{BFKMM17a,BFKMM17b}. We give more details in  Section~\ref{sec: 4th Moment}.

\subsection{Previous results}
We now collect previous bounds for $\Sq$ and $S_q(\balpha, \bbeta)$, which can be compared with our new bounds in Section~\ref{sec:New results}. This  aids in applications which require selecting the strongest bound in various ranges of parameters.
We note that for many applications the range $M,N\sim q^{1/2}$ is critical and this is where
our new bounds improve on all previous results.

We also refer to Section~\ref{sec:methods} on some comment concerning
on our approach and its on new features.  In Section~\ref{sec:gen}
we outline further possible generalisations and applications of our ideas.

To ease the comparisons, we now assume the coefficients $\balpha,\bbeta$ are both bounded and $a=1$. % $\gcd(a,q)=1$.

$\bullet$ Besides the trivial bound~\eqref{eq:trivial S}, we also have
\begin{equation}\label{eq:TypeI-complete-Poisson}
\Sq\ll Mq^{1+o(1)}
\end{equation}
by applying Poisson summation or equivalently invoking the P{\'o}lya--Vinogradov method. This classical approach also yields
\begin{equation}\label{eq:TypeII-complete-PolyaVinogradov}
S_q(\balpha, \bbeta)\ll MNq^{1/2+o(1)}\(q^{1/4}M^{-1/2}+q^{-1/4}+N^{-1/2}\)
\end{equation}
for the Type~II sums, see~\cite[Theorem~1.17]{FKM14}, which is non-trivial as long as $M>q^{1/2+\varepsilon}$ and $N>q^{\varepsilon}$ for any $\varepsilon>0$.
Improving the exponent $1/2$
is critical for most interesting applications.

$\bullet$ Kowalski, Michel and Sawin~\cite{KMS17,KMS20} employed the ``shift by $ab$" trick to deal with both of $\Sp$ and $S_{p}(\balpha, \bbeta)$. Driven by $\ell$-adic cohomology, the arguments in
\cite{KMS17,KMS20} work very well with all (hyper-) Kloosterman sums and general hyper-geometric sums without necessary obstructions. In the special case of Kloosterman sums, it is proven in~\cite{KMS17} that
$$ %\label{eq:TypeI-complete-KMS}
\Sp\ll MNp^{1/2+o(1)}\left(\frac{p^3}{M^2N^5}\right)^{1/12}
$$
and
$$ %\label{eq:TypeII-complete-KMS}
S_{p}(\balpha, \bbeta) \ll MNp^{1/2+o(1)}\(M^{-1/2}+(MN)^{-3/16}p^{11/64}\)
$$
with some mild restrictions on the sizes of $M$ and $N$. In particular, they succeed in obtaining non-trivial bounds in the P{\'o}lya--Vinogradov  range when $M,N=p^{1/2+o(1)}$.
Similar to~\cite{KMS20}, some variants of $\Sp$ and $S_{p}(\balpha, \bbeta)$ with an interval $\cI$ replaced by an arbitrary set $\cM$ have been estimated in~\cite{BaSh20} and~\cite{BS23},
respectively.

It is worth mentioning that Kowalski, Michel and Sawin~\cite{KMS20} are able to improve the trivial bound~\eqref{eq:trivial S} as
$$ %\label{eq:TypeII-complete-KMS2}
S_{p}(\balpha, \bbeta) \ll MNp^{1/2-\eta}
$$
for some $\eta>0$, provided that $M,N>p^\delta$ and $MN>p^{3/4+\delta}$ for any $\delta>0$.
The novelty here is that they do reach the exponent $3/4$, which is believed to be a classical barrier analogous to the 1/4 exponent which occurs in %%commB added 1/4
Burgess' bound for short character sums.

$\bullet$ Shparlinski and Zhang~\cite{SZ16}   proved
\begin{equation}\label{eq:TypeI-complete-ShparlinskiZhang1}
S_{q}
\({\mathbf 1}_\cI\) \ll (MN+q)q^{o(1)}
\end{equation}
for all primes $q$, as long as $\balpha= {\mathbf 1}_\cI$ is the indicator function of $\cI$.
This was shortly extended by Blomer, Fouvry, Kowalski, Michel and Mili\'cevi\'c~\cite{BFKMM17b} to non-correlations among Kloosterman sums and Fourier coefficients of modular forms. For an arbitrary $\balpha$, Shparlinski and Zhang~\cite{SZ16} proved
\begin{equation}\label{eq:TypeI-complete-ShparlinskiZhang2}
\Sq \ll \(MN\)^{1/2}q^{1+o(1)}
\end{equation}
for all primes $q$, and since the arguments in~\cite{SZ16} are completely elementary, one can easily check that  the bounds~\eqref{eq:TypeI-complete-ShparlinskiZhang1} and~\eqref{eq:TypeI-complete-ShparlinskiZhang2} can be identically extended to composite $q$.

$\bullet$ Shparlinski~\cite{Sh19} developed an elementary argument to show that
\begin{equation}\label{eq:TypeI-complete-Shparlinski}
\Sq \ll M^{3/4}\(N^{1/8}q+N^{1/2}q^{3/4}\)q^{o(1)}.
\end{equation}
This saves $q^{1/16}$ against the trivial bound in the P{\'o}lya--Vinogradov range $M,N = q^{1/2+o(1)}$, which is also utilized in~\cite{Sh19} to produce an asymptotic formula with a very sharp error term for second moments of twisted modular $L$-functions.
Note that~\eqref{eq:TypeI-complete-Shparlinski}
 also applies without any changes to more general bilinear
forms~\eqref{eq:Type-I-GenSet}.

$\bullet$ Xi~\cite{Xi17} developed an iteration process to produce large sieve inequalities of general trace functions. As a special consequence, it is proven that
$$
S_{q}(\balpha, \bbeta) \ll MNq^{1/2+o(1)}\(M^{-1}+N^{-1}+q^{-1}+(MN)^{-1}q\)
$$
for all squarefree $q$. This gives squareroot cancellations among Type~II sums in the
%% typical
 ``complete''  case $M=N=q$.

$\bullet$ In order to obtain applications to $\tau_3$ in arithmetic progressions, Xi~\cite{Xi18} combined the P{\'o}lya--Vinogradov method with arithmetic exponent pairs developed in~\cite{WX21}, and for any squarefree $q$ with all prime factors at most $q^\varepsilon$ for any $\varepsilon>0$,
it is proven that
$$
S_{q}(\balpha, \bbeta) \ll MNq^{1/2+o(1)}\(q^{\kappa/2}M^{(\lambda-\kappa-1)/2}+q^{-1/4}+N^{-1/2}\),
$$
where $(\kappa,\lambda)$ is (essentially) an arithmetic exponent pair defined as in~\cite{WX21}. In particular, the choice $(\kappa, \lambda)=(1/2,1/2)$ reproduces the above bound~\eqref{eq:TypeII-complete-PolyaVinogradov}. 

We emphasise that most of the above results can also be extended to
the sums $S_{q,a}(\balpha,\bbeta)$ and $S_{q,a}(\balpha)$.

We also note a new approach of Shkredov~\cite{Sh21}, which applies only to prime moduli $q = p$ but works for  very general sums. However in the case of the sums
$S_{p}(\balpha)$ and $S_{p}(\balpha, \bbeta)$ it produces results which so far have been
weaker than the best known. Furthermore, we stress that our approach seems to be the only
known way to obtain results for composite moduli $q$.

\subsection{Notation and conventions}
\label{sec:Not} 
We adopt the Vinogradov symbol $\ll$,  that is,
$$f\ll g~\Longleftrightarrow~f=O(g)~\Longleftrightarrow~|f|\le cg$$
for some absolute constant $c>0$. We also adopt $o$ notation
$$f=o(g)~\Longleftrightarrow |f|\le \varepsilon g,$$
for any $\varepsilon>0$ and sufficiently large values of parameters. Sometimes we will combine $\ll$ and $o$ notation and allow certain dependence on implied constants. In particular, an expression of the form
$$f\ll q^{o(1)}g$$
will mean that for all $\varepsilon>0$ there exists a constant $C_\varepsilon$ such that
$$|f|\le C_{\varepsilon}q^{\varepsilon}g,$$
for sufficiently large values of $q$ (and some other parameters).

It is convenient to write $n\sim N$ to indicate $N<n\leqslant2N$.

Throughout the paper, $p$ always denotes a prime number. For a finite set $\cS$ we use $\# \cS$ to denote its cardinality.

We also write 
$$
\e(z)=\exp(2\pi i z) \mand \eq(z)=\e(z/q)  \quad (\text{for}\ q\in\Z^+).
$$ 
We identify $\Z_q$ by the set $\{0, 1,\ldots, q-1\}$ and $\Z_q^*$ the subset consisting of all elements coprime to $q$,
and by $\{0,1,\ldots,p-1\}$ the finite field $\F_p$.
Hence an interval in $\F_p$ is understood as the set of the form
$$\{N_0+1 \bmod p,\ldots,N_0+N \bmod p\} \subseteq \{0,1,\ldots,p-1\},$$
that is, the set of residues modulo $p$ of some sequence of $N$ consecutive integers $N_0+1,\ldots,N_0+N$.

For the complex weight $\balpha=\{\alpha_m\}$ and $\sigma\ge 1$ we define the norms
$$
\|\balpha\|_\infty=\max_m|\alpha_m|  \mand \|\balpha\|_\sigma =\(\sum_m|\alpha_m|^\sigma\)^{1/\sigma}.
$$
For $g\in L^1(\R)$, define the Fourier transform
\[\widehat{g}(\lambda)=\int_\R g(x)\e(-\lambda x)\ud x.\]
In what follows, by a test function $\Phi$ we always mean a 
non-negative 
%% smooth function 
$\cC^\infty$-function (that is, a function having derivatives of all orders) 
which dominates the indicator
 function of $[-1,1]$. Using that $\Phi \in \cC^\infty$,  applying integration by parts, we have
\begin{equation}\label{eq:Fouriertransform-decay}
\widehat{\Phi}(y)\ll_A (1+|y|)^{-A}
\end{equation}
for any $A>0$, where in the above $\ll_A$ indicates that the implied constant depends on $A$.

%Finally, to simplify the notation, especially in the exponents, we write  $1/ab$ to mean the fraction of the form $1/(ab)$ rather than $b/a$ as the canonical convention requires.

\section{New results}
\label{sec:New results}

\subsection{Type~I sums of complete Kloosterman sums}\label{sec:results}
We first state our main results on upper bounds for $\Sqa$ with general $q$ and $a$.
\begin{theorem}\label{thm:TypeI-complete-q}
Let $q$ be a positive integer and let $\cI,\cJ$ be two intervals as in~\eqref{eq:Interval-IJ}. For any $M,N\geqslant 1$ and $a\in\Z$ with $d=\gcd(a,q)$, we have
$$
\Sqa\ll \|\balpha\|_2M^{1/2}Nq^{1/2+o(1)} \Delta_1(M,N,q,d),
$$
where we may take $\Delta_1(M,N,q,d)$ freely among
\begin{subequations}\begin{align}
&M^{-1/4}N^{-1}q^{1/2}d^{-1/4}+ q^{1/2}N^{-1}M^{-1/2} +N^{-1/2},\label{eq:Delta11}\\
&M^{-1/2}(N^{-3/4}q^{1/2}+d^{1/2})+N^{-1/2},\label{eq:Delta12}\\
&M^{-1/2}(N^{-1}q^{1/2}+(qd)^{1/4})+N^{-1/2}.\label{eq:Delta13}
\end{align}
\end{subequations}
\end{theorem}

We note that  around the ``diagonal'', that is, for $M = Nq^{o(1)}$ the bounds~\eqref {eq:Delta11} and~\eqref {eq:Delta12} coincide. In particular,  we note that  in the P{\'o}lya--Vinogradov  range $M,N = q^{1/2+o(1)}$ and $\|\balpha\|_\infty=q^{o(1)}$, the choice of the bound~\eqref {eq:Delta11}  in Theorem~\ref{thm:TypeI-complete-q}
yields $|\Sqa| \le q^{11/8 + o(1)}$ for $\gcd(a,q)=1$, and thus saves $1/8$ against the trivial bound
$q^{3/2+ o(1)}$, which is significantly better than the saving $1/24$ from~\cite{BFKMM17a} and $1/16$ from~\cite{Sh19}. On the other hand, the bound~\eqref {eq:Delta13} is better
than~\eqref {eq:Delta11} and~\eqref {eq:Delta12} for some skewed choices of $M$ and $N$, namely when
$$
MN^2 \ge q^{1+\delta} \mand M \ge q^{1/2+\delta}
$$
with some fixed $\delta>0$.

By virtue of the Selberg--Kuznetsov identity (see~\eqref{eq:Selberg-Kuznetsov} below) and M\"obius inversion, one may see the above three upper bounds also work for Type~I sums with $\cK_q(mn,1)$.

\begin{cor}\label{cor:TypeI-complete-q-afterSelbergKuznetsov}
Let $q$ be a positive integer and let $\cI,\cJ$ be two intervals as in~\eqref{eq:Interval-IJ}. For any $M,N\geqslant 1$, we have
$$
\sum_{m\in\cI} \sum_{n\in\cJ} \alpha_m\cK_q(mn,a)\ll \|\balpha\|_2M^{1/2}Nq^{1/2+o(1)} \Delta_1(M,N,q,1)
$$
uniformly in $a\in\Z_q^*$, where $\Delta_1(M,N,q,1)$ can be taken as in Theorem~\ref{thm:TypeI-complete-q}.
\end{cor}

%%\commX{(1.8) is not mentioned but I added (1.4)}

Figure~\ref{fig:OldNew-q} gives a plot of a polygon in the $(\mu, \nu)$-plane, with
$$
 M = q^{\mu+ o(1)} \mand N=q^{\nu + o(1)},
$$
where Theorem~\ref{thm:TypeI-complete-q}
improves the trivial bound~\eqref{eq:trivial S}, as well as~\eqref{eq:TypeI-complete-Poisson}, \eqref{eq:TypeI-complete-ShparlinskiZhang2} and~\eqref{eq:TypeI-complete-Shparlinski}.
Simple calculations show that the polygon on Figure~\ref{fig:OldNew-q} has vertices
$$
(\mu, \nu) = (0,1/2), \ (0,1), \ (1,1), \ (1,0), \ (1/2,1/4), \ (1/2,1/3), \ (2/5, 2/5).
$$

\begin{figure}[H]
\centering
 \includegraphics[scale=0.17]{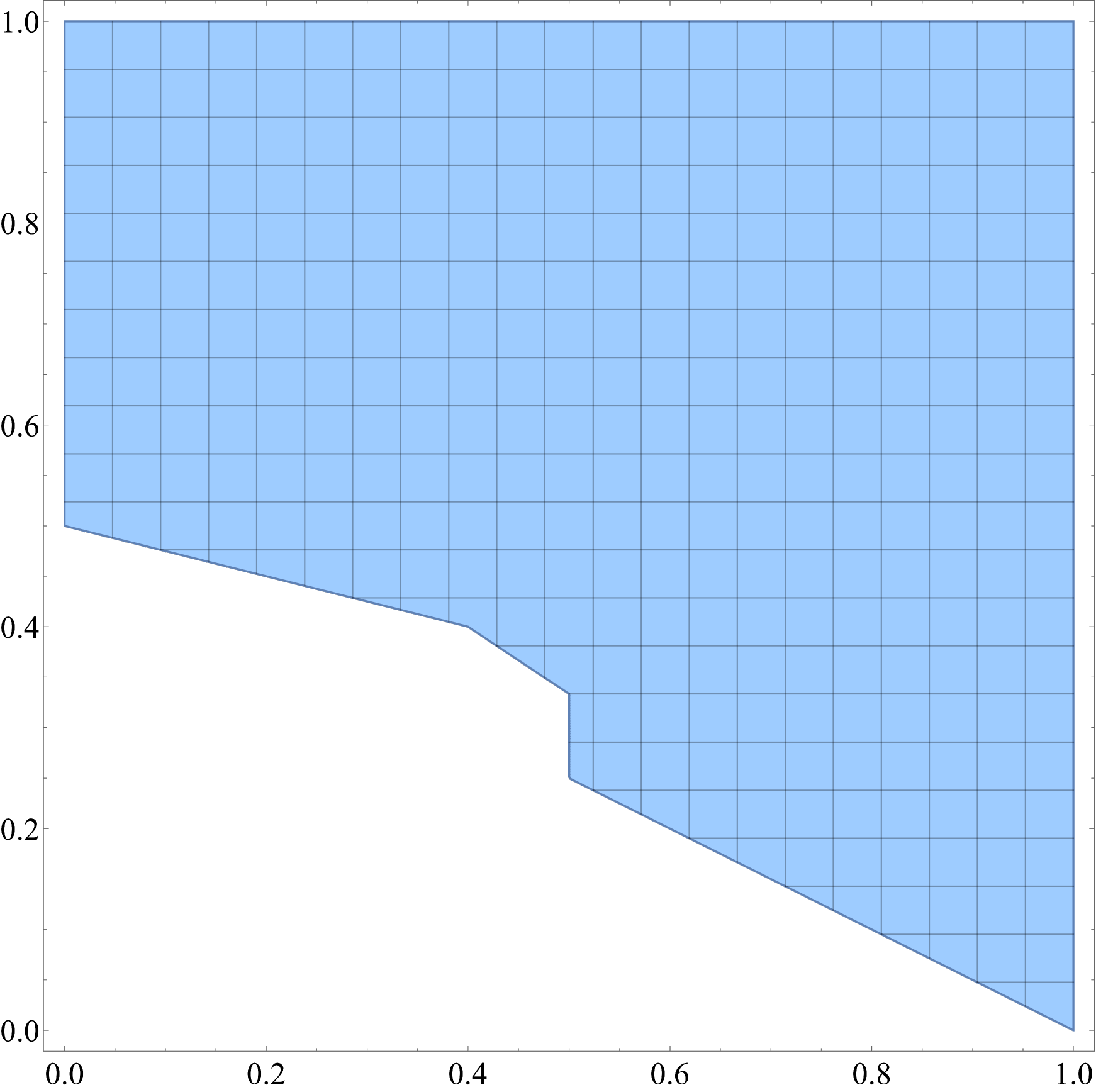}
 \caption{
The polygon where  Theorem~\ref{thm:TypeI-complete-q} wins.
}
 \label{fig:OldNew-q}
\end{figure}

\begin{rem}
Examining the proof of Theorem~\ref{thm:TypeI-complete-q} in Section~\ref{sec:TypeI-complete-proof}, one can easily see that it can be generalized identically to estimate Type~I sums of the weighted Kloosterman sum
$$
 \sum_{x \in \Z_q^*} \xi(x) \eq(m x +n \xbar)
$$
with an arbitrary complex weight $\xi$ such that $\|\xi\|_\infty \le 1$, to which the algebraic geometry method in~\cite{BFKMM17a, FKM14, KMS17, KMS20} does not apply.  In particular, our bound holds for
Sali{\'e} sums
$$
\cS_q(m,n) =  \sum_{x \in \Z_q^*} \(\frac{x}{q}\) \eq(m x +n \xbar)
$$
with all odd $q\geqslant1$, where $(\frac{\cdot}{q})$ denotes the Jacobi symbol mod $q$.
Bilinear forms with Sali{\'e} sums have been studied extensively in~\cite{DZ19,DKSZ20,KSSZ21,SSZ20,SSZ22} due to their applications to the moments
of $L$-functions of half-integral weight modular forms and to the distribution of modular square roots of primes.
In particular, the bounds in Theorem~\ref{thm:TypeI-complete-q}  applied to Sali{\' e} sums   improve~\cite[Theorem~2.2]{KSSZ21} for a wide range of parameters.
\end{rem}

\subsection{Type~II sums of incomplete Kloosterman sums}\label{sec:TypeII-incomplete-q}
The above estimates for Type~I sums $\Sq$ are essentially based on estimates for the following bilinear form with incomplete Kloosterman sums
$$
\WacIKqr=\sum_{m\in \cI}\sum_{\substack{k\in\Z_q^*\\ \langle ck\rangle_r\leqslant K}}\alpha_m \gamma_k \eq(a m \kbar),
$$
where $\bgamma$ is an arbitrary weight,
%%and we keep in mind $k\in\Z_q^*$ due to the appearance of $\kbar$
and we define $\langle n\rangle_r$ to be the (unique) integer $y\in[1,r]$ such that $n\equiv y\bmod r$. Hence a natural restriction is that $K\leqslant r$. Note that there are at most $Kq/r$ values of $k$ in the above average, thus a trivial bound could be
$$
|\WacIKqr|\leq \|\balpha\|_2 \|\bgamma\|_\infty M^{1/2}Kq r^{-1}
$$
for all $c\in\Z_r^*$ and $M,K\geqslant1$ with $K\leqslant r$.

\begin{theorem}\label{thm:TypeII-incomplete-q}
Let $q,r$ be positive integers with $r\mid q$, and $\cI$ an interval as in~\eqref{eq:Interval-IJ}.
For any $M,K\geqslant1$ with $K\leqslant r$, we have
$$
\WacIKqr\ll \|\balpha\|_2 \|\bgamma\|_\infty (M^{1/2}Kq^{1+o(1)}r^{-1} \Delta_2(M,K,q,r)
$$
uniformly in $a\in\Z_q^*$ and $c\in\Z_r^*$,
where we may take $\Delta_2(M,K,q,r)$ freely among
\begin{subequations}\begin{align}
&(Mq/r)^{-1/4}+M^{-1/2}+(Kq/r)^{-1/2},\label{eq:Delta21}\\
&M^{-1/2}(1+(K/r)^{-1/4}+K^{-1}r^{1/2})+(Kq/r)^{-1/2},\label{eq:Delta22}\\
&M^{-1/2}(1+K^{-1/2}r^{1/4}+K^{-1}r^{3/4})+(Kq/r)^{-1/2}.\label{eq:Delta23}
\end{align}
\end{subequations}
\end{theorem}

In the typical case $(r,c)=(q,1)$, for which we denote by $\WaIKq$ the corresponding 
bilinear sum,~\eqref{eq:Delta21} produces a non-trivial bound essentially in the full range $M,K\geqslant q^{\varepsilon}$, and~\eqref{eq:Delta22}  saves more than~\eqref{eq:Delta21} if $MK\gg q$. In turn,~\eqref{eq:Delta23} is  better than~\eqref{eq:Delta21}  and~\eqref{eq:Delta22} when $K$ is suitably large.

Besides applications to concluding Theorem~\ref{thm:TypeI-complete-q}, Theorem~\ref{thm:TypeII-incomplete-q} should be of independent interests and may indicate many other applications. In particular, Wu~\cite{Wu22} utilized such estimates for $\WaIKq$ to evaluate the fourth moment of Dirichlet $L$-functions with arbitrary moduli; see Section~\ref{sec: 4th Moment} for more details.

\subsection{New bilinear forms with arbitrary support}\label{sec:bilinearform-arbitraryset}
We now formulate some bilinear forms which generalize $\Sq$ or $\WaIKq$ to the case that one variable is supported on an arbitrary set. Since additive combinatorics in finite fields is employed, we now restrict our moduli to primes.

Let $\cK$ be an interval given by
\begin{equation}\label{eq:Interval-K}
\cK=\{K_0+1, \ldots, K_0+K\}
\end{equation}
for some integers $K_0$ and $K\geqslant1$. %\commI{$K_1 \to K_0$}
For each prime $p$ and an arbitrary set $\cM\subseteq\F_p$ of
cardinality $M$, define
$$
\WaMKp=\sum_{m\in\cM}\sum_{k\in\cK}\alpha_m\gamma_k\ep(am\kbar),
$$
where we also keep in mind the underlying restriction $k\in\F_p^*$.
In fact, it is contained implicitly in Shparlinski~\cite[Theorem~2.1]{Sh19} that
$$
\WaMKp\ll \|\balpha\|_\infty \|\bgamma\|_\infty  M^{3/4}(K^{7/8}p^{1/8}+K^{1/2}p^{1/4})p^{o(1)},
$$
which can also be generalized identically to the situation in $\Z_q$.

%%commB elaborated on the sentence just below
By combining the ``shift by $ab$'' trick of Vinogradov, Karatsuba and Friedlander--Iwaniec with techniques related to the Balog-Szemer\'{e}di-Gowers theorem, we obtain the following result.

\begin{theorem}\label{thm:TypeII-incomplete-arbitraryset1}
Let $p$ be a prime and let $r\in\Z^+$  be fixed. For an arbitrary set $\cM\subseteq\F_p$ of
cardinality $M$ and any interval $\cK\subseteq\F_p^*$ as in~\eqref{eq:Interval-K} of length  $K\geqslant p^{1/r}$, we have
$$
\WaMKp \ll\|\balpha\|_\infty \|\bgamma\|_\infty  MKp^{o(1)}
\(\frac{1}{M}+ \frac{p^{1+1/r}}{MK^2}\)^{7/24r}
$$
uniformly in $a\in \F_p^*$.
\end{theorem}

We now present another variant of Theorem~\ref{thm:TypeII-incomplete-arbitraryset1}.

\begin{theorem}\label{thm:TypeII-incomplete-arbitraryset2}
Let $p$ be a prime and let $r\in\Z^+$ be fixed. For an arbitrary set $\cM\subseteq\F_p$ of
cardinality $M\leqslant p^{1/2}$ and any interval $\cK\subseteq\F_p^*$ as in~\eqref{eq:Interval-K} of length  $K\geqslant p^{1/r}$, we have
$$
\WaMKp \ll \|\balpha\|_\infty \|\bgamma\|_\infty  MKp^{o(1)}
\(\frac{1}{M^{2r}}+\frac{1}{K}+\frac{p^{1+1/r}}{M^{10/13}K^2}\)^{1/4r}
$$
uniformly in $a\in \F_p^*$.
\end{theorem}

The upper bounds for $\Sqa$ benefit from those for $\WacIKqr$ as mentioned above. Following this spirit, one can imagine Theorems~\ref{thm:TypeII-incomplete-arbitraryset1} and~\ref{thm:TypeII-incomplete-arbitraryset2} can be utilized to study the following bilinear form of Kloosterman sums
\begin{equation}\label{eq:Type-I-GenSet}
\SMp=\sum_{m\in\cM} \sum_{n\in\cJ}\alpha_m\cK_p(m,n),
\end{equation}
where $\cM\subseteq \F_p$ is an arbitrary set and $\cJ$ is an interval given by~\eqref{eq:Interval-IJ}. The sums $\SMp$ have also been estimated in~\cite[Theorem~2.4]{BaSh20} (also with higher dimensional Kloosterman sums). Our new bounds are stronger and extends the
range in which nontrivial bounds of such sums are available.

\begin{theorem}\label{thm:TypeI-complete-arbitraryset1}
Let $p$ be a prime and let $r\in\Z^+$ be fixed.
For an arbitrary set $\cM\subseteq\F_p$ of
cardinality $M$ and any interval $\cJ\subseteq\F_p^*$ as in~\eqref{eq:Interval-IJ} of length  $N\leqslant p^{1-1/r}$, we have
$$
\SMp \ll \|\balpha\|_\infty   MNp^{1/2+o(1)} \(\frac{p^{1/2}}{M^{7/24r}N}+\frac{p^{1/2-7(r-1)/24r^2}}{M^{7/24r}N^{1-7/12r}}\).
$$
\end{theorem}

Slightly modifying the argument of the proof of
Theorem~\ref{thm:TypeI-complete-arbitraryset1} we also obtain the following bound.

\begin{theorem}\label{thm:TypeI-complete-arbitraryset2}
With the notation as in Theorem~\ref{thm:TypeI-complete-arbitraryset1} and assuming $M\leqslant p^{1/2}$, we have
\begin{align*}
\SMp \ll \|\balpha\|_\infty  M&Np^{1/2+o(1)} \\
& \times \(\frac{p^{1/2}}{M^{1/2}N}+\frac{ p^{1/2-1/4r}}{N^{1-1/4r}}+\frac{p^{1/2-(r-1)/4r^2}}{M^{13/40r}N^{1-1/2r}}\).
\end{align*}
\end{theorem}

Hence we beat the P{\'o}lya--Vinogradov barrier in the following sense.

\begin{cor}\label{coro:TypeI-complete-arbitraryset}
Let $\cM\subseteq\F_p$ be a set of
cardinality $M$ and let $\cJ\subseteq\F_p^*$ be an interval as in~\eqref{eq:Interval-IJ} of length $N$.
Then for any $\varepsilon>0$, there exist some $\delta,\eta > 0$ such that
$$
\SMp\ll\|\balpha\|_\infty  MNp^{1/2-\eta},
$$
provided that $M>p^\varepsilon$ and $N>p^{1/2-\delta}$.
\end{cor}

\subsection{General comments about the methods used}
\label{sec:methods}
In closing, we condense two new ingredients in this paper.

On one hand, all starting points have roots in analyzing the Type~II sum $\WacIKqr$, for which we reduce the problem to bounding $J_q(a,K)$ from above, individually and on average (see Section~\ref{sec:Jq(a,K)} for details). Regarding the averaged bound for $J_q(a,K)$, our tools include a variant of counting device due to Heath-Brown~\cite{HB78,HB86} and Cilleruelo and Garaev~\cite{CG11}. A new individual bound for $J_q(a,K)$ is based on squareroot cancellations among the products of two Kloosterman sums; see $T(x,t,z;q)$ as defined by~\eqref{eq:exponentialsum-T} and Lemma~\ref{lem:T-upperbound-q} for an upper bound.
These treatments of $J_q(a,K)$ form a vital part in this paper, and the relevant bounds should admit many applications on other occasions. We note that as in~\cite[Theorem~2.2]{Sh19} our method allows to obtain
stronger bounds for almost all $q$. In fact the first application of bounds on $J_q(a,K)$ to estimating bilinear forms with incomplete Kloosterman sums has been given by Heath-Brown~\cite[Section~3]{HB78}. In turn,  such estimates  have found a new application in the very recent work of Zhao~\cite{Zh22} on products of primes in arithmetic progressions.

On the other hand, we employ a recent result by Rudnev, Shkredov and Stevens~\cite{RSS20} on the decomposition of subsets of $\F_p$, with which we combine the ``shift by $ab$" trick of Vinogradov, Karatsuba and Friedlander--Iwaniec (see, for example~\cite{FI93,FM98,KMS17}) to study the new bilinear forms $\WaMKp$ and $\SMp$. Additive combinatorics then enters the picture and all details can be found in Section~\ref{sec:newbilinearforms-proofs}.
 Here we should also mention a pioneering work by Bourgain~\cite{Bo05}, who has employed tools from additive combinatorics (sum-product estimates) to study bilinear forms with incomplete Kloosterman sums, as well as a very recent work by Shkredov~\cite{Sh21} concerning bilinear forms with complete Kloosterman sums.

 Comparing with algebro-geometric methods, additive combinatorics could be an alternative approach, which sometimes can also provide very strong estimates. Note that although the results of Shkredov~\cite{Sh21}  do not seem to improve the best known estimates, there is certainly a strong potential for further  improvements. More precisely, the method of Shkredov~\cite{Sh21}
is based on the so-called  incidence bounds for hyperbolas. 
It is possible that the recent progress on such bounds due to Rudnev and Wheeler~\cite{RW22} can be used to improve the results within this approach. On the other hand, due to the nature
%%of this approach,
of the underlying technique, 
 it is not likely to work in residue rings modulo a composite number.

\subsection{Possible generalisations within our method}
\label{sec:gen}
To show its flexibility and without getting into details, we
note that at the cost of only simple typographical changes our bounds on the sums
$S_{p}(\balpha)$ extends (without any changes to Type~I bilinear forms
$$
\sum_{m \in \cI} \sum_{n \in \cJ} \alpha_m  \cK_q(\bxi, m,an) ,\qquad a \in \Z_q,
$$
 with more general sums
$$
\cK_q(\bxi, m,n) = \sum_{x \in \Z_q^*} \xi_x \eq(m x +n \xbar)
$$
with arbitrary complex weights $\bxi= \{\xi_x\}$, supported on $\Z_q^*$ and satisfying $\|\bxi\| \le 1$.

Another generalisation can involve the sum with more general binomials in the exponents,
that is, with $m x^s +n x^{-t}$ with some integers $s, t\in\Z$ instead of $mx +n \xbar$ as in~\cite{LSZ18b}. Such sums (with $(s,t) = (1, -2)$) appear in the investigation of the distribution of square-free numbers in arithmetic progressions, see~\cite{Nun16}.

\section{Applications}

\subsection{Moments of Dirichlet $L$-functions} \label{sec: 4th Moment}
In a recent work of Wu~\cite{Wu22}, our results have already been applied in deducing an asymptotic formula for the fourth moment of central values of Dirichlet $L$-functions. More precisely, he proved that
\begin{equation}\label{eq:fourthmoment}
\frac1{\varphi^*(q)}\mathop{\sum\nolimits^*}_{\chi \bmod q}\left|L\(\tfrac12,\chi\)\right|^4
=\prod_{p\mid q}\frac{(1-p^{-1})^3}{(1+p^{-1})}P_4(\log q)+O(q^{-\delta+o(1)})
\end{equation}
with $\delta=(1-6\vartheta)/14$, where $P_4$ is an explicit polynomial of degree $4$ and
$\vartheta$ denotes the exponent towards the Ramanujan--Petersson conjecture.
Here $q$ is a large positive integer with $q\not\equiv 2 \bmod4$, and $\chi$ runs over all primitive characters mod $q$ with $\varphi^*(q)$ denoting the number of all such characters.  Note that one may take $\vartheta=7/64$ thanks to
Kim and Sarnak~\cite{KiSa}.

Based on their deep observations in random matrix theory, Conrey, Farmer, Keating, Rubinstein and Snaith~\cite[\S 3.1 \& \S 4.3]{CFK+05} formulated the exact shape of $P_4$ and conjectured
that one may take $\delta=1/2$ in~\eqref{eq:fourthmoment}, which coincides with the square-root cancellation philosophy. The first instance establishing the existence of $\delta>0$ is due to Young~\cite{Yo11}, who proved that $\delta=(1-2\vartheta)/80$ is admissible for all large prime $q$.
In more recent works by Blomer, Fouvry,  Kowalski, Michel and Mili{\'c}evi{\'c}~\cite{BFKMM17a,BFKMM17b}, the dependence on the
Ramanujan-Petersson conjecture can be removed for prime moduli $q$, and one can even take a much better choice $\delta=1/20$. Note that new bounds on bilinear forms with Kloosterman sums~\cite{FKM14,KMS17,SZ16} play an important role
in their approach.

One of the ingredients in Wu~\cite{Wu22} is the application of our estimates for bilinear forms with incomplete Kloosterman sums as stated in Theorem~\ref{thm:TypeII-incomplete-q}, which play a crucial role in bounding off-diagonal terms with very unbalanced lengths of summations.
Ignoring the effect of the Ramanujan--Petersson conjecture (RPC), one notes that $\delta=(1-6\vartheta)/14$ in~\cite{Wu22} indicates the same power saving as in~\cite{BFKMM17b}, which shows $\delta=1/14$ for prime moduli under the RPC. Note that in~\cite{BFKMM17b},  a weaker result  with 
$\delta=1/16$ is claimed under the RPC, but
it turns out to be a mistake in calculations, which has been observed in~\cite[Theorem~1.1]{Zac19}.
 This means that, for general moduli, Theorem~\ref{thm:TypeII-incomplete-q} gains the same saving as the previous work for prime moduli. Power saving in~\cite{BFKMM17a,BFKMM17b} relies on an estimate of the following sum of Kloosterman sums
\begin{align*}
\mathop{\sum\sum}_{m_1,m_2,m_3,m_4}&W_1\left(\frac {m_1}{M_1}\right)W_2\left(\frac {m_2}{M_2}\right)\\
&\times W_3\left(\frac {m_3}{M_3}\right)W_4\left(\frac {m_4}{M_4}\right)\cK_p(m_1m_2,m_3m_4)
\end{align*}
with some compactly supported smooth functions $W_i(x)$, $i =1, \ldots , 4$. The estimate of this sum is based on both Type~I (see~\cite[Equation~(1.3)]{BFKMM17b}) and Type~II (see~\cite[Equation~(5.1)]{BFKMM17a}) sums of complete Kloosterman sums, which turns out to be quite different from 
Theorem~\ref{thm:TypeII-incomplete-q}. It is interesting to see that two different approaches lead to the same saving, 
while Theorem~\ref{thm:TypeII-incomplete-q} is available for general moduli and does not impose any smoothness conditions 
on the weights.

\subsection{Moments of short sums of Kloosterman sums}

A direct consequence of Theorem~\ref{thm:TypeI-complete-arbitraryset1} is the following bound on moments of  of short sums of  Kloosterman sums.

\begin{theorem}\label{thm:moments}
Let $r$ be a positive integer and let $\cJ\subseteq \F_p^{*}$ be an interval of length $N\leqslant p^{1-1/r}$.
For any fixed $\alpha\in[1,12r/7]$, we have
\begin{align*}
\sum_{\lambda \in \F_p^{*}}\left|\sum_{n\in \cJ}\cK_p(\lambda,n)\right|^{2\alpha}
&\ll p^{2\alpha+o(1)}N^{(12r-7\alpha)/(12r-7)}\\
&\qquad \qquad \times\left(1+\frac{N^2}{p^{1-1/r}}\right)^{7(\alpha-1)/(12r-7)}.
\end{align*}
\end{theorem}

In particular, taking $\alpha = 12r/7$ in Theorem~\ref{thm:moments} we have
\begin{equation}\label{eq:alpha=12r/7}
\sum_{\lambda \in \F_p^{*}}\left|\sum_{n\in \cJ}\cK_p(\lambda,n)\right|^{24r/7}
\ll p^{24r/7+o(1)}\left(1+\frac{N^2}{p^{1-1/r}}\right).
\end{equation}
In fact, we show in the proof of Theorem~\ref{thm:moments}  that this is essentially the
only case which has to be established.

As we have mentioned, the sums considered in Theorem~\ref{thm:TypeII-incomplete-arbitraryset1} are related to work of Bourgain~\cite{Bo05} on short bilinear sums of incomplete Kloosterman sums. In particular~\cite[Theorem~A.1]{Bo05} is based on~\cite[Lemma~A.2]{Bo05} which states that for an interval $I$ of length $M<p^{1/2}$ we have
\begin{equation}
\label{eq:largevalues111}
\#\left\{\lambda \in \F_p : ~\left|\sum_{x\in I}\e_p(\lambda x^{-1})\right|\ge p^{-\varepsilon}M\right\}\le p^{1+\delta}M^{-2}.
\end{equation}
Analysing the proof of~\cite[Lemma~A.2]{Bo05}, one sees that Bourgain's argument~\cite{Bo05} allows the choice
$$\delta=(4+o(1))\sqrt{\varepsilon}.$$
It is possible to use Theorem~\ref{thm:TypeII-incomplete-arbitraryset1}, with the set $\cM$  which appears on the left hand side of~\eqref{eq:largevalues111}, to show that one may take
$$\delta=\left(2\sqrt{\frac{24}{7}}+o(1)\right)\sqrt{\varepsilon}\approx 3.7032 \sqrt{\varepsilon}$$
in~\eqref{eq:largevalues111}.

\subsection{Divisor function in a family of arithmetic progressions}
\label{sec:DIv-AP}

For integers $a$ and $q\ge 2$  with $\gcd(a,q)=1$,
consider the
{\it divisor sum\/}
$$
\SXaq=\sum_{\substack{n\le X\\n\equiv a\bmod q}}\tau(n).
$$
By standard heuristic arguments, the expected main term  for $\SXaq$ should be
$$\MXaq =
\frac{\varphi(q) }{q^2}X \(\log X+2\gamma-1\)
- \frac{2}q X\sum_{d\mid q}\frac{\mu(d)\log d}d,
$$
where $\varphi(q)$ is the {\it Euler function of $q$\/}.
The basic problem is to prove that the error term
$$
\RXaq = \SXaq  - \MXaq
$$
can be bounded by $O(X^{1-\delta}/q)$ with some constant $\delta>0$ and $q$ being as large as possible compared to $X$.
This was realized independently by Selberg and Hooley~\cite{Ho57}, as long as
$q\le X^{2/3-\varepsilon}$ for any $\varepsilon>0$.

Banks,   Heath-Brown and  Shparlinski~\cite{BHBS05} and then
Blomer~\cite[Theorem~1.1]{Blom08} established bounds on the second
moment of error  term  $\RXaq$  which are nontrivial in the essentially optimal range
$q\le X^{1-\varepsilon}$ with an arbitrary fixed $\varepsilon>0$.

Kerr and Shparlinski~\cite{KeS20} have considered a mixed scenario between
pointwise and average bounds on $\RXaq$.  Namely, given a set $\cA\subseteq\Z_q^*$, we define
$$
  \EAR =  \sum_{a \in \cA} \RXaq.
$$
Then by~\cite[Theorem~1.2]{KeS20}, for any integers $A$, $X$ and $q$ with
$$ A\le q  \mand X \ge q \ge X^{19/31},
$$
and any interval $\cA \subseteq \Z_q^{*}$ of length $A$, we have
$$
|\EAR|  \le  ( A X^{1/2}   q^{-1/2 } + A^{1/8}  X^{1/4} q^{1/2} + A^{1/2}  X^{1/4} q^{1/4} ) q^{o(1)}.
$$
%which is nontrivial for $q^{1 - \varepsilon} \ge A \ge q^{3/7 + \varepsilon}$,
We refer to~\cite{KeS20}  for more details.

Theorem~\ref{thm:TypeI-complete-q} allows us to derive the following bound for $\EAR$.

\begin{theorem}\label{thm:EAN}
Let $q$ be a positive integer and $I$ an interval of length $A<q$.
For all $X\geqslant q$ and any set $\cA=I\cap \Z_q^*$ with $q^3<AX^2/8$, we have
$$
|\EAR|  \le (E(A,N,q)+A^{2/3}X^{1/3})X^{o(1)},
$$
where
\begin{align*}
E(A,N,q)=\min\{q^{1/2}X^{1/4}+X^{1/2} ,~X^{1/2}&(A^{1/4}+Aq^{-1/2}),\\
&\quad X^{1/2}(1+Aq^{-1/4})\}.
\end{align*}  
\end{theorem}

Theorem~\ref{thm:EAN} is nontrivial only for $q^3\ll AX^2$ and improves the results of~\cite{KeS20} for this range of parameters. If we could ignore the term $A^{2/3}X^{1/3}$, then Theorem~\ref{thm:EAN} would 
improve~\cite{KeS20} in the full range due to the appearance of
$q^{1/2}X^{1/4}$. In the crucial range $q\approx X^{2/3}$, Theorem~\ref{thm:EAN} beats the Selberg--Hooley barrier on average for $A\gg q^{3/8+\varepsilon}$ for any fixed $\varepsilon>0$, while one requires $A\gg q^{3/7+\varepsilon}$
in~\cite{KeS20}.

\section{Exponential sums}

\subsection{Basic properties of Kloosterman sums}
We first present some materials for Kloosterman sums, some of which can be found
in~\cite[Section~4.3]{Iwa97}. The following elementary properties are well-known and trivial to verify:
$$
\cK_q(m,n)=\cK_q(n,m),
$$
and
\begin{equation}\label{eq:Kloosterman-transition}
\cK_q(cm,n)=\cK_q(m,cn) \quad \text{if}~\gcd(c,q)=1.
\end{equation}
We also need the following Selberg--Kuznetsov identity
\begin{equation}\label{eq:Selberg-Kuznetsov}
\cK_q(m,n)=\sum_{d\mid \gcd(m,n,q)}d \cK_{qd^{-1}}(mnd^{-2},1).
\end{equation}
Kloosterman sums also enjoy the twisted multiplicativity
$$
\cK_q(m,n)=\cK_{q_1}(mq_2^{-1},nq_2^{-1})\cK_{q_2}(mq_1^{-1},nq_1^{-1})
$$
for any $q_1,q_2$ satisfying $q=q_1q_2$ and $\gcd(q_1,q_2)=1$.

The Kloosterman sum $\cK_q(m,n)$ reduces to the Ramanujan sum
$$
c_q(m)=\sum_{x\in\Z_q^*}\eq(m x)
$$
if $n\equiv0 \bmod q$, see~\cite[\S~3.2]{IK04}  for a background on
Ramanujan sums. In particular, we frequently use the inequality
\begin{equation}\label{eq:Ramanujansum-upperbound}
|c_{q}(m)| \leqslant \gcd(q,m).
\end{equation}

\subsection{Fourier transforms of products of Kloosterman sums}
By virtue of Kloosterman sums, we define
\begin{equation}\label{eq:exponentialsum-T}
T(x,y,z;q)=\frac{1}{q}\sum_{t\in\Z_q}\cK_q(x,t)\cK_q(y,t)\eq(-zt).
\end{equation}
This can be regarded as the (normalized) Fourier transform of products of two Kloosterman sums. For the purpose of estimating bilinear forms in this paper, we would like to explore an upper bound for $T(x,y,z;q)$ which exhibits squareroot cancellations up to some harmless factors.

The Chinese Remainder Theorem, together with the twisted multiplicativity of Klooterman sums yields the following twisted multiplicativity
\begin{equation}\label{eq:T-multiplicativity}
T(x,y,z;q)=T(x{q}^{-1}_2,y{q}^{-1}_2,z;q_1)T(x{q}^{-1}_1,y{q}^{-1}_1,z;q_2)
\end{equation}
for any $q_1,q_2$ satisfying $q_1q_2=q$ and $\gcd(q_1,q_2)=1$. Hence the evaluation of
$T(x,y,z;q)$ reduces to the situation of prime power moduli.

Here is our upper bound for $T(x,y,z;q)$.
\begin{lemma}\label{lem:T-upperbound-q}
Let $q$ be a positive integer. For $x,y,z\in\Z$, we have
$$
T(x,y,z;q)\ll \gcd(x,y,q)^{1/2}\gcd\(x-y,z,\frac{q}{(x,y,q)}\)^{1/2}q^{1/2+o(1)}.
$$
\end{lemma}

The proof of Lemma~\ref{lem:T-upperbound-q} is given in Appendix~\ref{sec:Kloostermanappendix}, since the arguments are based on vanishings and exact expressions of Kloosterman sums, which might be of independent interests.

\section{Counting with reciprocals}\label{sec:Jq(a,K)}

\subsection{Notation}
For integers $a$ and $q$, we are interested in the counting function
\begin{equation}\label{eq:Jq(a,K)}
J_q(a,K)=\#\{(k_1, k_2)\in[1,K]^2:~k_1^{-1}-k_2^{-1}\equiv a \rmod q\},
\end{equation}
which turns out to be an important object and occupies a central role in our later arguments.
There is a trivial bound
%%\begin{equation}\label{eq:Jq(a,K)-trivialbound}
$$
J_a(a,K)\leqslant K(K/q+1)
$$
%%\end{equation}
for all $K>1$, which reaches the true order of magnitude as $K\geqslant q$. The situation now becomes subtle as long as $K<q$, in which case we need,  for the sake of later applications, to beat the above bound in as wide a range of parameters as possible.

\subsection{Pointwise bounds}
The following lemma gives a uniform bound for $J_q(a,K)$, which is in fact contained implicitly in a series of papers of Heath-Brown; see~\cite[Page~367]{HB78} or~\cite[Page~46]{HB86} for instance.

\begin{lemma}\label{lem:Jq(a,K)-Heath-Brown}
Let $q$ be a positive integer and let $a\in\Z$. For any $K\leqslant q$,
we have
$$
J_q(a,K)\le \(\frac{K^{3/2}}{q^{1/2}}+\frac{K^2\gcd(a,q)}{q}+1\)q^{o(1)}.
$$
\end{lemma}

By virtue of explicit evaluations of Kloosterman sums, we also have the following alternative upper bound for $J_q(a,K)$, which is sharper than Lemma~\ref{lem:Jq(a,K)-Heath-Brown} as long as $K>q^{2/3}$.

\begin{lemma}\label{lem:Jq(a,K)-new}
Let $q$ be a positive integer and  let $a\in\Z$. For any $K\leqslant q$,
we have
$$
J_q(a,K)\le \(\frac{K^2}{q}+\frac{K\gcd(a,q)^{1/2}}{q^{1/2}}+q^{1/2}\)q^{o(1)}.
$$
\end{lemma}

\begin{proof}
Let $\Phi$ be a non-negative smooth function which dominates the indicator function of $[-1,1]$.
Trivially we write
$$
J_q(a,K)\ll \mathop{\sum\sum}_{{k}^{-1}_1-{k}^{-1}_2\equiv a\bmod q}
\Phi\(\frac{k_1}{K}\)\Phi\(\frac{k_2}{K}\).
$$
Applying Poisson summation to each of $k_1, k_2 \bmod{q}$, we find
$$
J_{q}(a,K)\ll\frac{K^2}{q^2}
\mathop{\sum\sum}_{m,n\in\Z}T(-m,n,a;q)\widehat{\Phi}\(\frac{mK}{q}\)\widehat{\Phi}\(\frac{nK}{q}\),
$$
where $T(-m,n,a;q)$ is defined by~\eqref{eq:exponentialsum-T}.
In view of the decay~\eqref{eq:Fouriertransform-decay}, we may truncate the $m$ and $n$-sums at $H=q^{1+\varepsilon}/K$ with an arbitrary $\varepsilon>0$ at a cost of a negligible error term. Hence
$$
J_{q}(a,K)\ll\frac{K^2}{q^2}
\mathop{\sum\sum}_{0\leqslant |m|,|n|\leqslant H}|T(m,n,a;q)|+1.
$$
The result now follows immediately from Lemma~\ref{lem:T-upperbound-q}.
\end{proof}

\subsection{Bounds on average}
We may also improve Lemmas~\ref{lem:Jq(a,K)-Heath-Brown} and~\ref{lem:Jq(a,K)-new} on average, provided that the length of averaging is not too short.

\begin{lemma}\label{lem:Jq(a,K)-onaverage}
Let $q$ be a positive integer and let  $\gcd(a,q)=1$. For any positive integers $N,K\leqslant q$, we have
$$
\sum_{1\leqslant n\leqslant N} J_q(an,K)\le \(\frac{K^2N^{1/2}}{q^{1/2}}+K\)q^{o(1)}.
$$

\end{lemma}

\begin{proof}
Denote by $J$ the sum in question, which counts the number of solutions to
$$
k_1^{-1}- k_2^{-1}\equiv an \bmod q
$$
with $1\leqslant k_1,k_2\leqslant K$ and $1\leqslant n \leqslant N$. Moreover, denote by $I$ the number of solutions to
$$
k_1^{-1}- k_2^{-1}\equiv a(n_1-n_2) \bmod q
$$
in $1\leqslant k_1,k_2\leqslant K$ and $1\leqslant n_1,n_2\leqslant 2N$.
Clearly, we have
\begin{equation}\label{eq:J-I}
J\ll I/N
\end{equation}
since for each $1\leqslant n\leqslant N$ we have
$$
\#\{(n_1,n_2) \in [1, 2N]^2:~ n=n_1-n_2\} \gg N.
$$

The problem now reduces to bounding $I$ from above. For each $\lambda\in\Z_q$, define
\begin{equation}\label{eq:I(lambda)}
I(\lambda)=\#\{(k,n)\in [1,K]\times [1,2N]:~k^{-1}+an\equiv \lambda \bmod q\},
\end{equation}
so that
\begin{equation}
\label{eq:Idef1}
I=\sum_{\lambda \in \Z_q}I(\lambda)^2.
\end{equation}

Since
\begin{equation}\label{eq:I(lambda)-l1norm}
\sum_{\lambda \in \Z_q}I(\lambda)\ll KN,
\end{equation}
it suffices to produce an upper bound for $I(\lambda)$ uniformly in $\lambda \in \Z_q$.

Recall the definition~\eqref{eq:I(lambda)}. It is useful to note that $I(\lambda)$ counts the number of solutions to the congruence
\begin{equation}\label{eq:I(lambda)-2}
(ak)^{-1}+n\equiv a^{-1}\lambda \bmod q, \  1\leqslant k \leqslant K, \   1\leqslant n \leqslant 2N.
\end{equation}
We now adopt some ideas of Cilleruelo and Garaev~\cite[Theorem~1]{CG11}. By the Dirichlet pigeonhole principle, we  can choose integers $t$ and $u$ satisfying
$$
a^{-1}\lambda u\equiv t\bmod q, \quad 0\leqslant |t|\leqslant T, \quad 0 < |u|\leqslant \frac{2q}{T},
$$
where $T>2$ is a parameter to be determined later.
In this way, we infer from~\eqref{eq:I(lambda)-2} that
$$
k(t-nu)\equiv a^{-1}u\bmod q.
$$
Denote by  $a^{-1}u\equiv b\bmod q$ with $1\le b\le q-1$.
We conclude there exists some integer $r$ such that
$$
0\neq k(t-nu)=b+qr,\quad |r|\leqslant 2R
$$
with
$$
R=\frac{KT+KNq/T+q}{q}=1+\frac{KT}{q}+\frac{KN}{T}.
$$

Collecting the above arguments, we derive that
$$
I(\lambda)
\leqslant \mathop{\sum\sum\sum}_{\substack{k\leqslant K,~n\leqslant2N,~|r|\leqslant 2R\\ k(t-nu)= b+qr} }1
\ll\sum_{|r|\leqslant 2R}\tau(|b+qr|)\le Rq^{o(1)}.
$$
%where we choose an appropriate $b$ with $1\leqslant b\leqslant q$ such that $a^{-1}u\equiv b\bmod q$ and $b+qr\neq0$.
Taking $T=(qN)^{1/2}$, we find
$$
I(\lambda)\le (1+K(N/q)^{1/2})q^{o(1)},
$$
from which and~\eqref{eq:I(lambda)-l1norm} it follows that
$$
I\le KN(1+K(N/q)^{1/2})q^{o(1)}.
$$
This readily proves the lemma in view of~\eqref{eq:J-I}.
\end{proof}

We note that Karatsuba~\cite[Lemma~8]{Kar}  has previously estimated  the quantity $I$ defined by~\eqref{eq:Idef1} as
\begin{equation}\label{eq:I-Kar}
I \le KN(1 + N/K +K^{1/2}  N^{3/2} q^{-1/2})q^{o(1)},
\end{equation}
which implies a sharper bound compared to Lemma~\ref{lem:Jq(a,K)-onaverage} for smaller values of $N$. However, for such ranges of parameters, the bound~\eqref{eq:I-Kar} is comparable to an application of Lemma~\ref{lem:Jq(a,K)-Heath-Brown}.

\subsection{A slight extension}
Let $q$ be a positive integer and $r\mid q.$ Define
\begin{equation}\label{eq:Jq(a,K;r,c)}
\begin{split}
J_q(a,K&; r,c)\\
& =\#\left\{(k_1,k_2)\in(\Z_q^*)^2:
\begin{array}{c}
k_1^{-1}-k_2^{-1}\equiv a \rmod q,\\
\langle ck_1\rangle_r,\langle ck_2\rangle_r\leqslant K
\end{array}\right\}.
\end{split} 
\end{equation}

Clearly, we have $J_q(a,K;q,1)=J_q(a,K)$ for all $K\leqslant q$.
In general, we may conclude the following inequality, illustrating an upper bound for $J_q(a,K;r,c)$ in terms of the original counting function~\eqref{eq:Jq(a,K)}.

\begin{lemma}\label{lem:Jq(a,K;r,c)}
Let $q$ be a positive integer and $r\mid q$. For all $K\geqslant1$ and $c\in\Z_r^*,$ we have
$$
J_q(a,K;r,c)\leqslant \frac{q}{r}J_r(c^{-1}a,K).
$$
\end{lemma}
\proof
Note that $J_q(a,K;r,c)$ is exactly the number of quadruples $(u_1,u_2,v_1,v_2)$ with
$$
u_1,u_2\bmod r,~v_1,v_2\bmod{q/r},~\gcd((u_1+rv_1)(u_2+rv_2),q)=1,
$$
$$
(u_1+rv_1)^{-1}-(u_2+rv_2)^{-1}\equiv a \bmod q,
$$
$$
\langle cu_1\rangle_r,\langle cu_2\rangle_r\leqslant K.
$$
The congruence condition also indicates that
$u_1^{-1}-u_2^{-1}\equiv a \bmod r.$
For each fixed $u_2,u_2\bmod r$ and $v_1\bmod{q/r}$, there should exist at most one
$v_2\bmod{q/r}$ with
$$
\gcd(u_2+rv_2,q)=1, ~
(u_1+rv_1)^{-1}-(u_2+rv_2)^{-1}\equiv a \bmod q.
$$
Therefore, we infer
$$
J_q(a,K;r,c)\leqslant \frac{q}{r}\#\left\{(u_1,u_2)\in(\Z_r^*)^2:
\begin{array}{c}
u_1^{-1}-u_2^{-1}\equiv a \bmod r,\\
\langle cu_1\rangle_r,\langle cu_2\rangle_r\leqslant K
\end{array}\right\}.
$$
Since $c\in\Z_r^*$, we complete the proof after a change of variable,
\endproof

\section{Bilinear forms with Kloosterman sums over intervals}

\subsection{Preliminary comments}
 Following the approach of~\cite{Sh19}, we relate Type~I
bilinear forms with complete  Kloosterman sums to Type~II
bilinear forms with incomplete  Kloosterman sums. Hence we first prove
Theorem~\ref{thm:TypeII-incomplete-q}, from which Theorem~\ref{thm:TypeI-complete-q} is derived.

\subsection{Proof of Theorem~\ref{thm:TypeII-incomplete-q}}\label{sec:TypeII-incomplete-proof}
Without loss of generality, we assume $\bgamma$ is bounded.
By the Cauchy--Schwarz inequality
$$
|\WacIKqr|^2
\leqslant \|\balpha\|_2^2\sum_{m\in \Z}\Phi\(\frac{m-M_0}{M}\)\left|\sum_{\langle ck\rangle_r\leqslant K}
\gamma_k\eq(amk^{-1})\right|^2,
$$
where $\Phi$ is a non-negative smooth function which dominates the indicator function of $[-1,1]$
(and $M_0$ is as in~\eqref{eq:Interval-IJ}).
Squaring out and interchanging summations we get
$$
|\WacIKqr|^2
\leqslant \|\balpha\|_2^2\mathop{\sum\sum}_{\langle ck_1\rangle_r,\langle ck_2\rangle_r\leqslant K}
\gamma_{k_1}\overline{\gamma}_{k_2}\Sigma(k_1,k_2)
$$
with
$$
\Sigma(k_1,k_2)=\sum_{m\in \Z}\Phi\(\frac{m-M_0}{M}\)\eq(am(k_1^{-1}-k_2^{-1})).
$$

From Poission summation it follows that
$$
\Sigma(k_1,k_2)=M\sum_{n\equiv a(k_2^{-1}-k_1^{-1})\bmod q}\widehat{\Phi}\(\frac{Mn}{q}\)\eq(-nM_0).
$$
In view of the decay~\eqref{eq:Fouriertransform-decay}, the contributions from those $n$ with $|n|>q^{1+\varepsilon}/M$ are negligible, so that
$$
\Sigma(k_1,k_2)
=M\sum_{\substack{0\leqslant |n|\leqslant q^{1+\varepsilon}/M\\ n\equiv a(k_2^{-1}-k_1^{-1})\bmod q}}\widehat{\Phi}\(\frac{Mn}{q}\)\eq(-nM_0)+O(K^{-2}),
$$
giving
$$
\WacIKqr^2
\ll\|\balpha\|_2^2M\sum_{1\leqslant |n|\leqslant q^{1+\varepsilon}/M}J_q\(a^{-1}n,K;r,c\)+\|\balpha\|_2^2MKq/r,
$$
where $J_q(a^{-1}n,K;r,c)$ is defined as~\eqref{eq:Jq(a,K;r,c)}.
Note that the first term does not appear unless $M\leqslant q^{1+2\varepsilon}$, which we henceforth assume.
In view of Lemma~\ref{lem:Jq(a,K;r,c)}, we further have
\begin{align*}
\WacIKqr^2
\ll \|\balpha\|_2^2Mq/r\sum_{1\leqslant |n|\leqslant q^{1+\varepsilon}/M}J_r\((ac)^{-1}n,K\)& \\
+\|\balpha\|_2^2M&Kq/r.
\end{align*}  

To conclude upper bounds for $\WacIKqr$, it suffices to invoke individual or averaged estimates for $J_r(\cdot,K)$. We proceed on a case by case basis depending on if $q^{1+\varepsilon}/M\le r$ or not.

Suppose first that $q^{1+\varepsilon}/M>r$. We have
$$
\WacIKqr^2
\ll \|\balpha\|_2^2\frac{q^{2+\varepsilon}}{r^2}\sum_{n=1}^{r}J_r((ac)^{-1}n,K)+\|\balpha\|_2^2MKq/r.
$$
%It follows from~\eqref{eq:Jq(a,K)-trivialbound} that 
From the trivial observation 
$$
\sum_{n=1}^{r}J_r((ac)^{-1}n,K) = 
\sum_{n=1}^{r}J_r(n,K)  \le K^2
$$
we conclude that 
$$
\WacIKqr^2
\ll \|\balpha\|_2^2\frac{q^{2+\varepsilon}K^2}{r^2}+\|\balpha\|_2^2MKq/r.
$$
Due to the presence of terms $M^{-1/2}$ in~\eqref{eq:Delta21},~\eqref{eq:Delta22} and~\eqref{eq:Delta23} we see that Theorem~\ref{thm:TypeII-incomplete-q} is satisfied in this case.

Suppose next that $q^{1+\varepsilon}/M\le r$ and note $K\leqslant r$.
Since $\varepsilon>0$ is arbitrary, it follows from
Lemmas~\ref{lem:Jq(a,K)-onaverage}  and~\ref{lem:Jq(a,K;r,c)} that
$$
|\WacIKqr|^2
\le \|\balpha\|_2^2(K^2M^{1/2}(q/r)^{3/2}+MKq/r)q^{o(1)},
$$
which gives the choice~\eqref{eq:Delta21}. The remaining two choices~\eqref{eq:Delta22} and~\eqref{eq:Delta23} are produced by applying Lemmas~\ref{lem:Jq(a,K)-Heath-Brown}
and~\ref{lem:Jq(a,K)-new} in place of
%%\ccc{Lemma~\ref{lem:Jq(a,K;r,c)},}
Lemma~\ref{lem:Jq(a,K)-onaverage},
 respectively.

We now complete the proof of Theorem~\ref{thm:TypeII-incomplete-q}.

\subsection{Proof of Theorem~\ref{thm:TypeI-complete-q}}\label{sec:TypeI-complete-proof}

We now prove Theorem~\ref{thm:TypeI-complete-q}
as a consequence of Theorem~\ref{thm:TypeII-incomplete-q}. We assume $N\leqslant q$, otherwise we may choose $\Delta_1(M,N,q,d)=q^{1/2}/N$, which is better than $N^{-1/2}$.

Following the approach in~\cite[Theorem~2.1]{Sh19} we put $I=\rf{\log (N/2)}$ and consider $2(I+1)$  sets
\begin{equation}\label{eq:Xi,pm}
\begin{split}
\cX_0&=\{x\in\Z_q^*:a_1x\equiv y\bmod{q_1}\text{~for some~}y\in[1,q_1/N]\}, \\
\cX_i& = \{x\in\Z_q^*: a_1x\equiv y\bmod{q_1}\\
& \qquad \qquad \qquad \qquad \text{~for some~}y\in(e^{i-1}q_1/N,e^iq_1/N]\},
\end{split}
\end{equation}
for $i = 1, \ldots, I$, where $a_1=a/d$, $q_1=q/d$  with $d=\gcd(a,q)$.
By~\cite[Equation~(5.3)]{Sh19} we have
$$
\Sqa\ll \sum_\pm\sum_{0\leqslant i\leqslant I}|S_i^\pm|,
$$
where for $i=0,\ldots,I$,
$$
S_i^{\pm}=\sum_{m\in\cI} \sum_{x \in \cX_i} \alpha_m\gamma_x\eq(\pm m\xbar)
$$
with some  complex weights $\gamma_x$ satisfying
$$
\gamma_x \ll e^{-i} N, 
$$
and $\sum_\pm$ means that sum includes both $S_i^+$ and $S_i^-$. 
Each sum $S_i ^{\pm}$ is of the type $ \WacIKqr$
with $K= e^{i} q_1/N$ and $r=q_1$.
We are then in a position to apply Theorem~\ref{thm:TypeII-incomplete-q}, and the first choice in~\eqref{eq:Delta21} yields
$$
S_i^{\pm}\ll  \|\balpha\|_2 M^{1/2}q((dM)^{-1/4}+(N/q)^{1/2})q^{o(1)}.
$$

Summing over $i$, we obtain the choice~\eqref{eq:Delta11} of $\Delta_1(M,N,q,d)$ for Theorem~\ref{thm:TypeI-complete-q}.
The other two choices~\eqref{eq:Delta12} and~\eqref{eq:Delta13} of $\Delta_1(M,N,q,d)$ correspond to applications of Theorem~\ref{thm:TypeII-incomplete-q} with the remaining choices~\eqref{eq:Delta22} and~\eqref{eq:Delta23} of
$\Delta_2(M,K,q,r)$.

\section{Some results from additive combinatorics}
\label{sec:newbilinearforms-proofs}

\subsection{Preliminary comments}
We next present some preliminaries for the proof of Theorems~\ref{thm:TypeII-incomplete-arbitraryset1} and~\ref{thm:TypeII-incomplete-arbitraryset2}. The results of this section are specialised to $\F_p$ instead of $\Z_q$ for two reasons. The first is our use of various results from additive combinatorics (see Lemma~\ref{lem:MPSS}) and the second is that we require uniform estimates for counting zeros of polynomials in one variable over finite fields (see~\eqref{eq:muui} below).

 To begin with, we state a series of results on the decomposition of subsets and counting problems in $\F_p$.

\subsection{Energy and special partitions of sets}
For $\cA,\cB\subseteq \F_p$, as usual we define the \emph{difference set}
$$
\cA-\cB=\{a-b:~(a,b)\in\cA\times\cB\}
$$
and the \emph{additive energy}
$$
E(\cA ,\cB)=\#\{(a_1,a_2,b_1,b_2)\in  \cA^2 \times \cB^2:~a_1-a_2=b_1-b_2\}.
$$
We also write  $E(\cA,\cA)=E(\cA)$ for abbreviation.

We quote a special case of  Rudnev,  Shkredov and   Stevens~\cite[Theorem~2.15]{RSS20} with $k=2$ therein.

\begin{lemma}\label{lem:RSS}
Let $L\geqslant1$. For any subset $\cA\subseteq \F_p$ with
$$
E(\cA)\ge \frac{(\#\cA)^3}{L},
$$
there must exist some $\cB\subseteq \cA$ and $\cD\subseteq \cA-\cA$ satisfying
$$
\# \cB\ge \frac{\# \cA}{16L} \mand  \# \cD\le 16L\# \cA,
$$
such that for each $b\in \cB$ we have
$$
\# \{ (a, d)\in \cA\times  \cD:~b=a-d\}\ge \frac{\# \cA}{4L}.
$$
\end{lemma}

Lemma~\ref{lem:RSS} guarantees the existence of $\cB\subseteq\cA$ and $\cD\subseteq \cA-\cA$ such that $\cA\cap(b+\cD)$ is suitably large for any $b\in\cB$, provided that the energy $E(\cA)$ can be bounded from below. This means a certain number of elements in $\cA$ can be represented by $b+d$ with $d\in\cB$ and $d\in\cD$, which  creates more room for possible cancellations in bilinear forms to be studied later.

Iterating Lemma~\ref{lem:RSS} gives the following decomposition of an arbitrary set $\cA\subseteq \F_p$.

\begin{lemma}\label{lem:RSS-Iter}
For any $L\geqslant1$ and $\cA\subseteq \F_p$, there exist disjoint sets $\cA_0,\cA_1,\dots,\cA_J\subseteq \F_p$ satisfying
$$
\cA=\bigsqcup_{0\leqslant j\leqslant J}\cA_j
$$
with $J\ll L$, such that both of the following hold:
\begin{enumerate}[$(i)$]
\item
$$
E(\cA_0)\ll\frac{(\# \cA)^3}{L},\quad \min_{1\leqslant j\leqslant J}\#\cA_j\gg\frac{\#\cA}{L};
$$
\item
for each $1\leqslant j\leqslant J$, there exists $\cD_j\subseteq \cA-\cA$ satisfying
$$
\#\cD_j\ll L\#\cA,
$$
such that for each $a_j\in \cA_j$, we have
\begin{equation}\label{eq:shiftedcounting-lowerbound}
\# \{ (a,d_j) \in \cA \times \cD_j :~a_j=a-d_j\}\gg \frac{\#\cA}{L}.
\end{equation}
\end{enumerate}
\end{lemma}

\begin{proof}
Given an arbitrary $L\ge 1$, we assume
$$
E(\cA)\ge\frac{(\#\cA)^3}{L},
$$
since otherwise the proposition holds trivially with $J=0$ and $\cA_0=\cA$. By Lemma~\ref{lem:RSS}, there exist
$\cA_1\subseteq\cA$ and $\cD_1\subseteq \cA-\cA$ such that
$$
\#\cA_1\gg\frac{\#\cA}{L}
$$
and (ii) also holds for $j=1$.

We now write $\cB_1 =\cA\setminus\cA_1$, and assume that
$$
E(\cB_1 )\ge\frac{(\#\cA)^3}{L},
$$
since otherwise the proposition holds with $J=1$ and $\cA_0=\cB_1 $.
We now apply Lemma~\ref{lem:RSS} with
$$
\cA\leftarrow \cB_1 ,\quad L\leftarrow L\(\frac{\#\cB_1 }{\#\cA}\)^3,
$$
getting that there exist some $\cA_2\subseteq\cB_1 $ and $\cD_2\subseteq \cB_1 -\cB_1 $ satisfying
$$
\# \cA_2\gg \frac{\# \cB_1 }{L} \(\frac{\#\cA}{\#\cB_1 }\)^3\geqslant \frac{\# \cA}{L}
$$
and
$$
\# \cD\ll  L\(\frac{\#\cB_1 }{\#\cA}\)^3\# \cB_1 \leqslant L\#\cA,
$$
such that for each $a_2\in \cA_2$ we have
$$
\# \{ (a, d_2)\in \cB_1 \times  \cD_2:~a_2=a-d_2\}\gg \frac{\# \cB_1 }{L} \(\frac{\#\cA}{\#\cB_1 }\)^3\geqslant \frac{\# \cA}{L},
$$
which yields~\eqref{eq:shiftedcounting-lowerbound} trivially.
Further iterating, if $\cA_2\neq \cB_1 $, completes the proof.
\end{proof}

\subsection{Counting with products}

For a set $\cA\subseteq \F_p$, define
$$
D_p(\cA)=\#\{(a_1,\ldots, a_8)\in\cA^8:(a_1-a_2)(a_3-a_4) =(a_5-a_6)(a_7-a_8)\}.
$$
This quantity mixes the addition and multiplication in $\Fp$, and has been studied widely in additive combinatorics. Trivially we have
$$
D_p(\cA)\leqslant (\#\cA)^7.
$$
The best known bound up to now is due to Macourt, Petridis, Shkredov and Shparlinski~\cite[Theorem~4.3]{MPSS20}.

\begin{lemma}\label{lem:MPSS}
For any subset $\cA \subseteq \F_p$ of cardinality $A\leq p^{1/2}$,  we have
$$
D_p(\cA)\ll A^{84/13}p^{o(1)}.
$$
\end{lemma}
We recall, the following bound on the number of solutions of multiplicative
congruences, which follows from  a result of
Ayyad, Cochrane and Zheng~\cite[Theorem~1]{ACZ96}; see also Kerr~\cite{Ke17} for a stronger statement.

\begin{lemma}\label{lem:product-ACZKerr}
Suppose $\cA,\cB\subseteq\F_p^*$ are two intervals with $\#\cA=A$ and $\#\cB=B$. Then we have
$$
\mathop{\sum\sum\sum\sum}_{\substack{(a_1,a_2,b_1,b_2)\in\cA^2\times\cB^2\\ a_1b_1= a_2b_2}}1
\le \(\frac{AB}{p}+1\)ABp^{o(1)}.
$$
\end{lemma}

We remark that using the result of Cochrane and Shi~\cite[Theorem~2]{CS10} one can extend Lemma~\ref{lem:product-ACZKerr}  to congruences with arbitrary moduli.

\begin{lemma}\label{lem:product-energy-ACZKerr}
Suppose $\cA,\cB\subseteq\F_p^*$ are two intervals with $\#\cA=A$ and $\#\cB=B$. For any subset $\cC\subseteq\F_p$, we have
$$
\mathop{\sum\sum\sum\sum}_{\substack{(a_1,a_2,b_1,b_2)\in\cA^2\times\cB^2\\ a_1b_1=a_2b_2}}E(a_1\cC,a_2\cC)
\le E(\cC) \(\frac{AB}{p}+1\)ABp^{o(1)}.
$$
\end{lemma}

\begin{proof}
The Cauchy--Schwarz inequality implies
$$
E(a_1\cC,a_2\cC)^2
\leqslant E(a_1\cC) E(a_2\cC)=E(\cC)^2
$$
for any $a_1,a_2\in\F_p^*$. The lemma follows directly from Lemma~\ref{lem:product-ACZKerr}.
\end{proof}

We now derive a bound for hybrid counting problems as an alternative of Lemmas~\ref{lem:MPSS} and~\ref{lem:product-ACZKerr}.

\begin{lemma}\label{lem:MPSS-product-ACZKerr-mixed}
Suppose $\cA\subseteq\F_p^*$ is an interval with $\#\cA=A$, and $\cC\subseteq\F_p$ is an arbitrary subset with $\#\cC=C$. Denote by $N$ the number of solutions
$(a_1,a_2,c_1,c_2,c_3,c_4)\in\cA^2\times\cC^4$ to
$$
a_1(c_1-c_2)=a_2(c_3-c_4)\in\F_p^*.
$$
For $C\leqslant p^{1/2}$, we have
\begin{equation}\label{eq:N-asymptotic}
N =\frac{A^2C^4}{p}+O(AC^{42/13}p^{o(1)}).
\end{equation}
\end{lemma}

\begin{proof}  Denote by $\cX$ be the set of all multiplicative characters in $\F_p^*$ and
by $\cX^*$ the set of all non-trivial ones. Orthogonality yields
$$
N=\frac{1}{p-1}\sum_{\chi\in\cX}\left|\sum_{a\in\cA} \chi(a)\right|^2 \left|\mathop{\sum\sum}_{c_1,c_2\in \cC} \chi(c_1-c_2)\right|^2.
$$
The contribution from the trivial character is
$$
\frac{A^2C^2(C-1)^2}{p-1}=\frac{A^2C^4}{p}+O\(\frac{A^2C^3}{p}\),
$$
which gives the main term in~\eqref{eq:N-asymptotic}. Denote by $N_1$ the remaining contribution. The Cauchy--Schwarz inequality implies
$$
N_1^2 \leqslant \frac{1}{p-1}\sum_{\chi\in\cX^*}\left| \sum_{a\in\cA} \chi(a)\right|^4 \times \frac{1}{p-1}\sum_{\chi\in\cX^*} \left|\mathop{\sum\sum}_{c_1,c_2\in \cC}
\chi(c_1-c_2)\right|^4.
$$
The estimate of Ayyad, Cochrane and Zheng~\cite[Theorem~2]{ACZ96} yields
$$
\frac{1}{p-1}\sum_{\chi\in\cX^*} \left|\sum_{a\in\cA} \chi(a)\right|^4\leqslant A^2p^{o(1)},
$$
from which we obtain
$$
N_1^2\leqslant D_p(\cC) A^2p^{o(1)}.
$$
Then the error term in~\eqref{eq:N-asymptotic} readily follows from Lemma~\ref{lem:MPSS}.
\end{proof}

\section{Bilinear forms with Kloosterman sums over arbitrary sets}
\subsection{Proof of Theorem~\ref{thm:TypeII-incomplete-arbitraryset1}}
Without loss of generality, we assume $\balpha,\bgamma$ are both bounded by $1$.
Define
$$
\delta=\frac{1}{M}+\frac{p^{1+1/r}}{MK^2},
$$
and note that we may assume $\delta<1$ since otherwise Theorem~\ref{thm:TypeII-incomplete-arbitraryset1} becomes trivial.

Let $L\ge 1$ be some parameter and apply Lemma~\ref{lem:RSS-Iter} to decompose $\cM$ as the union of disjoint subsets  $\cA_0,\cA_1,\dots,\cA_J\subseteq \cM$, which admit the same properties as in Lemma~\ref{lem:RSS-Iter}.
In this way, we may write
\begin{equation}
\label{eq:W-sumSj}
\WaMKp=\sum_{o\leqslant j\leqslant J}S_j,
\end{equation}
where
\begin{equation}
\label{eq:Sj}
S_j=\sum_{m\in \cA_j}\sum_{k\in\cK}\alpha_m\gamma_k\ep(amk^{-1}).
\end{equation}
We estimate $S_0$ and $S_j$ ($1\leqslant j\leqslant J$) by different methods.

The Cauchy--Schwarz inequality gives
\begin{equation}\label{eq:S0-Sigma}
|S_0|^2 \leqslant  K \Sigma,
\end{equation}
where
$$
\Sigma=\sum_{m_1, m_2 \in \cA_0}\left| \sum_{k\in\cK}\ep(a(m_1-m_2) k^{-1})\right|.
$$
Note that $\cK$ is an interval given by~\eqref{eq:Interval-K}. We introduce the ``shift by $ab$'' trick as in~\cite{FI93,FM98,KMS17}, which implies that for some $\xi\in\R$,  we have
%% and obtain
$$
\Sigma\ll \frac{\log p}{UV}\sum_{m_1, m_2 \in \cA_0}\sum_{k\in\cH}\sum_{v\sim V}
\left|  \sum_{u\sim U}\e(\xi u)  \ep(a(m_1-m_2)(k + uv)^{-1})\right|
$$
%%for some $\xi\in\R$, 
where $\cH$ is another
interval of length at most $2K$, and $U,V$ are to be optimized later subject to the constraint $UV\leqslant K$
(we recall the definition of  $\e(z) = \exp(2 \pi i z)$ in Section~\ref{sec:Not}).

To group variables, we put
\begin{equation}\label{eq:R(lambda,mu)}
R(\lambda,\mu)=\#\left\{(k,m_1,m_2,v):\begin{array}{c}
a(m_1-m_2)=v\lambda\in\F_p,\\
k=v\mu\in\F_p,\\
k\in\cH, m_1, m_2 \in \cA_0,v \sim V
\end{array}\right\}\end{equation}
for $\lambda,\mu\in\F_p$.
Now we can write
$$
\Sigma\ll\frac{\log p}{UV}\mathop{\sum\sum}_{\lambda,\mu\in\F_p}R(\lambda, \mu)\left|  \sum_{u\sim U}e(\xi u)
\ep(\lambda (\mu + u)^{-1})\right| .
 $$
By H{\"o}lder's inequality, we derive that
\begin{equation}\label{eq:Sigma-123}
\Sigma\ll\frac{\log p}{UV} \Sigma_1^{1-1/r} (\Sigma_2\Sigma_3)^{1/2r},
\end{equation}
where
\begin{align*}
\Sigma_1& =\mathop{\sum\sum}_{\lambda,\mu\in\F_p}R(\lambda, \mu),\\\Sigma_2&=\mathop{\sum\sum}_{\lambda,\mu\in\F_p}R(\lambda, \mu)^2,\\
\Sigma_3 &=\mathop{\sum\sum}_{\lambda,\mu\in\F_p}\left|\sum_{u\sim U}e(\xi u)
\ep(\lambda (\mu + u)^{-1})\right|^{2r}.
 \end{align*}

Clearly,
\begin{equation}\label{eq:Sigma1}
\Sigma_1\ll M^2KV,
\end{equation}
and $\Sigma_2$ counts all solutions to the system of congruences
\begin{align*}
(m_1-m_2)v_2&= (m_3-m_4)v_1\in\F_p,\\
k_1v_2&= k_2v_1\in\F_p
\end{align*}
in $k_1,k_2 \in \cH$, $m_1, m_2, m_3,m_4 \in \cA_0$ and $v_1,v_2 \sim V$. In the language of additive energies, we also have
$$
\Sigma_2=\mathop{\sum\sum\sum\sum}_{\substack{k_1,k_2\in\cH,~v_1,v_2\sim V\\ k_1v_2\equiv k_2v_1\bmod p}}E(v_1\cA_0,v_2\cA_0).
$$
From Lemma~\ref{lem:product-energy-ACZKerr} and the prescribed bound
$$
E(\cA_0)\ll\frac{M^3}{L}
$$
as in Lemma~\ref{lem:RSS-Iter}, we infer
\begin{equation}\label{eq:Sigma2}
\Sigma_2\le \frac{M^3KV}{L}\(\frac{KV}{p}+1\)p^{o(1)}.
\end{equation}

Appealing to the orthogonality of additive characters, we find
$$\Sigma_3\le p\widetilde \Sigma_3,
$$
where
$\widetilde \Sigma_3$ counts all solutions to the equation
\begin{equation}
\label{eq:muui}
\sum_{1\leqslant i\leqslant 2r} (-1)^i (\mu + u_i)^{-1}=0
\end{equation}
in $u_1,\ldots, u_{2r}\sim U$ and $\mu\in\F_p$. If the  set $\{u_1, \ldots, u_{2r}\}$ can be partitioned into $r$ pairs of equal elements  $u_i = u_j$,  then there are $p$ possible values for $\mu$ and thus the total contribution from such solutions is $O(U^{r}p)$. For other choices of $u_1,\ldots,u_{2r}$
there are obviously at most $2r-1$ possible values of $\mu$ and thus to total contribution from such solutions is $O(U^{2r})$.
Therefore, we find
\begin{equation}\label{eq:Sigma3}
\Sigma_3\leqslant p\widetilde \Sigma_3\ll p(U^{r}p + U^{2r}).
\end{equation}

Inserting the bounds~\eqref{eq:Sigma1}, \eqref{eq:Sigma2} and~\eqref{eq:Sigma3} to~\eqref{eq:Sigma-123}, we obtain
$$
\Sigma^{2r}
 \le\frac{(M^2KV)^{2r-2}}{(UV)^{2r}}\frac{M^3KV}{L}\(\frac{KV}{p}+1\)(U^{r}p + U^{2r})p^{1+o(1)}.
$$
Taking
\begin{equation}\label{eq:U,V-choice}
U=p^{1/r},\quad V=K/U,
\end{equation}
we find $V\geqslant1$ since we assume $K\geqslant p^{1/r}$,
and the above bound can be simplified to
$$
\Sigma^{2r}\le \frac{M^{4r-1}K^{2r-2}}{L}(K^2+p^{1+1/r})p^{o(1)},
$$
which after the substitution in~\eqref{eq:S0-Sigma} implies that
\begin{equation}\label{eq:S0-upperbound}
S_0\ll MKL^{-1/4r}p^{o(1)}\delta^{1/4r}  .
\end{equation}

We now turn to consider $S_j$ for $1 \leqslant j\leqslant J$. Recalling~\eqref{eq:Sj}, by the Cauchy--Schwarz inequality
\begin{equation}
\label{eq:Sj-Pij}
|S_j|^2\leqslant\# \cA_j  \Pi_j,
\end{equation}
where
$$
\Pi_j=\sum_{m\in \cA_j} \left|\sum_{k\in\cK}\gamma_k\ep(amk^{-1})\right|^2.
$$
In view of Lemma~\ref{lem:RSS-Iter}~(ii), we obtain
$$
\Pi_j\ll \frac{L}{M}\sum_{m\in\cM}\sum_{d\in \cD_j}\left|\sum_{k\in\cK}\gamma_k\ep(a(m-d)k^{-1})\right|^2.
$$
Squaring out and switching summations, it follows that
$$
\Pi_j \ll \frac{L}{M}\mathop{\sum\sum}_{k_1,k_2\in\cK}
\left| \sum_{m\in\cM}\ep(am(k_1^{-1}-k^{-1}_2))\right| \left| \sum_{d\in \cD_j}\ep(ad(k_1^{-1}-k^{-1}_2)) \right| .
$$
Again by the Cauchy--Schwarz inequality, we infer
\begin{equation}\label{eq:Pij-Pij1Pij2}
\Pi_j \ll \frac{L}{M} (\Pi_{j,1} \Pi_{j,2})^{1/2}
\end{equation}
with
\begin{align*}
\Pi_{j,1}&=\mathop{\sum\sum}_{k_1,k_2\in\cK}
\left|\sum_{m\in\cM}\ep(am(k_1^{-1}-k^{-1}_2))\right|^2,\\
\Pi_{j,2}&=\mathop{\sum\sum}_{k_1,k_2\in\cK}
\left|\sum_{d\in \cD_j}\ep(ad(k_1^{-1}-k^{-1}_2))\right|^2.
\end{align*}

The treatments to $\Pi_{j,1}$ and $\Pi_{j,2}$ are quite similar, and we only give the details for the former one. In fact,
$$
\Pi_{j,1}=\mathop{\sum\sum}_{m_1,m_2\in\cM}
\left|\sum_{k\in\cK}\ep(a(m_1-m_2)k^{-1})\right|^2.
$$
Following the above arguments, the ``shift by $ab$'' trick in~\cite{FI93,FM98,KMS17} yields
$$
\Pi_{j,1}\ll\frac{K\log p}{U^2V}\mathop{\sum\sum}_{m_1,m_2\in\cM}
\sum_{k\in\cH}\sum_{v\sim V}\left|\sum_{u\sim U}e(\xi u)\ep(a(m_1-m_2)(k+uv)^{-1})\right|^2
$$
for some $\xi\in\R$ and $U,V\geqslant$ with $UV\leqslant K$, where $\cH$ is an
interval of length at most $2K$. Therefore,
$$
\Pi_{j,1}\ll\frac{K\log p}{U^2V}\mathop{\sum\sum}_{\lambda,\mu\in\F_p}
T(\lambda,\mu)\left|\sum_{u\sim U}e(\xi u)\ep(\lambda (\mu+u)^{-1})\right|^2,
$$
where
$T(\lambda, \mu)$ counts all solutions to the system of equations
$$
a(m_1-m_2) = v\lambda\in\F_p,\qquad  k= v\mu\in\F_p
$$
in $k\in\cH$, $m_1,m_2\in\cM$ and $v\sim V$. Following the previous arguments regarding $S_0$, we would like to apply H\"older's inequality, for which the first and second moments of $T(\lambda, \mu)$ need to be under control. In fact,
$$
\mathop{\sum\sum}_{\lambda,\mu\in\F_p}T(\lambda,\mu)\ll M^2KV
$$
and
\begin{align*}
\mathop{\sum\sum}_{\lambda,\mu\in\F_p}T(\lambda,\mu)^2
&\le E(\cM)KV(KV/p + 1)p^{o(1)}\\
&\le M^3KV(KV/p + 1)p^{o(1)}
\end{align*}
by Lemma~\ref{lem:product-energy-ACZKerr} and the trivial bound $E(\cM)\leqslant M^3$.
Therefore, H\"older's inequality yields
\begin{align*}
(\Pi_{j,1})^r
\le\frac{(MK)^{2r-1}}{U^{2r}V} &\(\frac{KV}{p}+1\)p^{o(1)}\\
& \times   \mathop{\sum\sum}_{\lambda,\mu\in\F_p} \left| \sum_{u\sim U}e(\xi u)\ep(\lambda (\mu+u)^{-1})\right|^{2r}.
\end{align*}
As argued above, the last sums over $\lambda,\mu,u$ contribute at most $p(U^{r}p+U^{2r})$, so that
\begin{equation}\label{eq:Pij1-upperbound}
(\Pi_{j,1})^r
\le (MK)^{2r-1}\(K+\frac{p^{1+1/r}}{K}\)p^{o(1)}
\end{equation}
upon the choice of $U$ and $V$ as in~\eqref{eq:U,V-choice}.
%%commB choices of -> choice of
A similar argument shows
\begin{equation}\label{eq:Pij2-upperbound}
(\Pi_{j,2})^r \le (LMK)^{2r-1}\(K+\frac{p^{1+1/r}}{K}\)p^{o(1)},
\end{equation}
for which we use $E(\cD)\leqslant (LM)^3$ in view of Lemma~\ref{lem:RSS-Iter}~(ii).

Combining~\eqref{eq:Pij1-upperbound}, \eqref{eq:Pij2-upperbound} and~\eqref{eq:Pij-Pij1Pij2}, we obtain
$$
\Pi_j \le MK^2L^{2-1/2r}p^{o(1)} \delta^{1/r}.
$$
Summing over $j$, we derive from~\eqref{eq:Sj-Pij} that
$$
\sum_{1\leqslant j\leqslant J}|S_j|^2\leqslant M^2K^2L^{2-1/2r}p^{o(1)} \delta^{1/r}.
$$
Note that $J\ll L$ as in Lemma~\ref{lem:RSS-Iter}, hence Cauchy--Schwarz yields
\begin{equation}\label{eq:sumSj-upperbound}
\sum_{1\leqslant j\leqslant J}|S_j|\leqslant MKL^{3/2-1/4r}p^{o(1)} \delta^{1/2r}.
\end{equation}

Inserting~\eqref{eq:S0-upperbound} and~\eqref{eq:sumSj-upperbound} into~\eqref{eq:W-sumSj}, we obtain
$$
\WaMKp\ll MKp^{o(1)}(L^{-1/4r} \delta^{1/4r}+L^{3/2-1/4r} \delta^{1/2r}).
$$
Now Theorem~\ref{thm:TypeII-incomplete-arbitraryset1} follows immediately by choosing $L=\delta^{-1/6r}$.

\subsection{Proof of Theorem~\ref{thm:TypeII-incomplete-arbitraryset2}}
We also assume $\balpha,\bgamma$ are both bounded and proceed as in the proof of Theorem~\ref{thm:TypeII-incomplete-arbitraryset1}, however we do not decompose the set $\cM$.
Similarly,  for any $U,V\geqslant1$ with $UV\leqslant K/2$, we have
\begin{equation}\label{eq:W-Xi}
|\WaMKp|^2\leqslant K \Xi,
\end{equation}
where
$$
\Xi \ll\frac{\log p}{UV}\mathop{\sum\sum}_{\lambda,\mu\in\F_p} R(\lambda,\mu) \left| \sum_{u\sim U}e(\xi u) \ep(\lambda (\mu + u)^{-1})\right|
$$
for some $\xi\in\R$,
where $R(\lambda, \mu)$ is defined as in~\eqref{eq:R(lambda,mu)} with $\cA_0$ replaced by $\cM$.

As argued above, we also need to bound the first and second moments of $R(\lambda,\mu)$ when applying H\"older's inequality.
In particular, we would like to bound the second moment by virtue of Lemma~\ref{lem:MPSS-product-ACZKerr-mixed} in place of Lemma~\ref{lem:product-energy-ACZKerr}. To do so, we pick out the contribution from the terms with $\lambda=0$, so that
\begin{equation}\label{eq:Xi-Xi1}
\Xi \ll \Xi_1+MKp^{o(1)}
\end{equation}
with
$$
\Xi_1=\frac{\log p}{UV}\sum_{\lambda\in\F_p^*}\sum_{\mu\in\F_p} R(\lambda,\mu) \left| \sum_{u\sim U}e(\xi u) \ep(\lambda (\mu + u)^{-1})\right|.
$$
Now H\"older's inequality yields
\begin{equation}\label{eq:Xi1-Sigma*}
\Xi_1^{2r}\le \frac{(M^2KV)^{2r-2}}{(UV)^{2r}}(pU^r+U^{2r})p^{1+o(1)} \Sigma^*,
\end{equation}
where $\Sigma^*$ counts all solutions to the system of equations
\begin{align*}
(m_1-m_2)v_2&= (m_3-m_4)v_1\in\F_p^*,\\
k_1v_2&= k_2v_1\in\F_p
\end{align*}
in $k_1,k_2\in\cH$ with $\#\cH\leqslant 2K$, $m_1,m_2,m_3,m_4 \in \cM$, $v_1,v_2 \sim V$
with  $m_1\neq m_2$ and $m_3 \neq m_4$.
From Lemma~\ref{lem:MPSS-product-ACZKerr-mixed} it follows that
$$
\Sigma^*\le K\(\frac{M^4V^2}{p}+M^{42/13}V\)p^{o(1)}.
$$
After a substitution in~\eqref{eq:Xi1-Sigma*} we infer
$$
\Xi_1^{2r}
\le M^{4r}K^{2r}\(\frac{1}{K}+\frac{p^{1+1/r}}{M^{10/13}K^2}\)p^{o(1)},
$$
and combined with~\eqref{eq:Xi-Xi1} we obtain
%%from which and -> and combined with
$$
\Xi \le M^2K\(\frac{1}{K}+\frac{p^{1+1/r}}{M^{10/13}K^2}\)^{1/2r}p^{o(1)}+MKp^{o(1)}
$$
Inserting this into~\eqref{eq:W-Xi}, Theorem~\ref{thm:TypeII-incomplete-arbitraryset2} follows readily.

\subsection{Proofs of Theorems~\ref{thm:TypeI-complete-arbitraryset1} and~\ref{thm:TypeI-complete-arbitraryset2}}
We may argue as in the proof of Theorem~\ref{thm:TypeI-complete-q}, and it suffices to bound the Type~II sum
$$
Y=\sum_{m\in\cM} \sum_{x \in \cX_{i,\pm}} \alpha_m\gamma_x\ep(m\xbar)
$$
for $0\leqslant i\leqslant \rf{\log (N/2)}$, where $\cX_{i,\pm}$ is defined by~\eqref{eq:Xi,pm} with $q$ replaced by $p$.
Note that $N\leqslant p^{1-1/r}$ guarantees $e^{i} p /N \geqslant p^{1/r}$ for each $i$, and we conclude from Theorem~\ref{thm:TypeII-incomplete-arbitraryset1} that
$$
Y\ll \|\balpha\|_\infty   Mp^{1+o(1)}\(\frac{1}{M}+\frac{N^2}{ M p^{1-1/r}}\)^{7/24r}.
$$
This proves Theorem~\ref{thm:TypeI-complete-arbitraryset1}, and the proof of Theorem~\ref{thm:TypeI-complete-arbitraryset2} can also be completed if employing Theorem~\ref{thm:TypeII-incomplete-arbitraryset2} instead of Theorem~\ref{thm:TypeII-incomplete-arbitraryset1}.

\section{Proofs of applications}
\subsection{Proof of Theorem~\ref{thm:moments}}

It suffices to prove~\eqref{eq:alpha=12r/7}, and the general case follows from H\"older's inequality. To this end,
we put
$$
\fm(\lambda)=\sum_{n\in \cJ}\cK_p(\lambda,n)
$$
and consider the $2\alpha$-th moment
$$
\fM_{\alpha}=\sum_{\lambda \in \F_p^*}|\fm(\lambda)|^{2\alpha}.
$$
H\"older's inequality yields
$$
\fM_\alpha
\leqslant \fM_1^{\alpha_1}(\fM_{12r/7})^{7\alpha_2/12r}
$$
with $\alpha_1+\alpha_2=\alpha$,  $\alpha_1+7\alpha_2/12r=1$ and $\alpha_1,\alpha_2\geqslant0$.
Equivalently, we have
$$
\alpha_1=\frac{12r-7\alpha}{12r-7}\geqslant0,\mand \alpha_2=\frac{12r(\alpha-1)}{12r-7}\geqslant0
$$
since $1\leqslant \alpha\leqslant 12r/7$.
Note that the orthogonality of additive characters gives
$$
\fM_1\leqslant\sum_{\lambda \in \F_p}|\fm(\lambda)|^2\ll p^2N.
$$
Hence
$$
\fM_\alpha
 \leqslant(p^2N)^{(12r-7\alpha)/(12r-7)}(\fM_{12r/7})^{7(\alpha-1)/(12r-7)},
$$
and the general case then follows from~\eqref{eq:alpha=12r/7} readily.

We now turn to prove~\eqref{eq:alpha=12r/7}. Performing a dyadic decomposition and pigeonhole principle, there exist some $V>0$ and a subset $\Lambda\subseteq \F_p^*$ given by
$$
\Lambda=\{\lambda\in \F_{p}^*:~V\leqslant|\fm(\lambda)|<2V\},
$$
such that
\begin{equation}
\label{eq:moments1}
\fM_{12r/7}\ll p^{o(1)}V^{24r/7}\#\Lambda.
\end{equation}
On the other hand,
$$
\sum_{\lambda \in \Lambda}|\fm(\lambda)|\geqslant  V\#\Lambda,
$$
from which and Theorem~\ref{thm:TypeI-complete-arbitraryset1} it follows that
$$
V(\#\Lambda)^{7/24r}\ll p^{1+o(1)}\left(1+\frac{N^{7/12r}}{p^{7(r-1)/24r^2}} \right).
$$
Alternatively, we have
$$
V^{24r/7}\#\Lambda\ll p^{24r/7+o(1)}\left(1+\frac{N^2}{p^{1-1/r}}\right),
$$
from which and~\eqref{eq:moments1} we obtain~\eqref{eq:alpha=12r/7}, completing the proof of Theorem~\ref{thm:moments}.

\subsection{Proof of Theorem~\ref{thm:EAN}}

We essentially follow the proof of~\cite[Theorem~1.2]{KeS20} but apply Theorem~\ref{thm:TypeI-complete-q}   instead of~\eqref{eq:TypeI-complete-Shparlinski} as used in~\cite{KeS20}.

We now sketch some details.
First we fix some sufficiently small $\varepsilon > 0$ and for each positive divisor $d\mid q$ we define
$$
U(d) = d^2 X^{-1}  \mand  V(d) = d^2 XY^{-2}(qX)^{2\varepsilon},
$$
where $Y\in[1,X/2]$ is to be chosen later.
In particular, we have
$$
U(d) \le V(d).
$$

Then by~\cite[Equation~(2.9)]{KeS20}
 \begin{equation}
\label{eq: EAB Eab}
 |\EAR| \le \EaR+ O\(A(Y/q + 1)(Yq)^{\eps}\),
 \end{equation}
 where
\begin{equation}
\label{eq: EaR}
\EaR =  \frac{1}{q} \sum_{d \mid q} \sum_\pm  \left|  \sum_{a \in \cA}
\sum_{1\leqslant n\leqslant V(d)}  w_d(n) \cK_d(\pm n, a)\right|
 \end{equation}
with some (explicit) weights $w_d(n)$ satisfying
$$
|w_d(n) | \le
 \begin{cases} X^{1+\eps+o(1)} d^{-1}, & \text{if\ } n \le U(d), \\
 X^{1/4+o(1)}  d^{1/2} n^{-3/4}, & \text{if\ }  U(d) < n \le V(d).
 \end{cases}
$$
Next, for each $d \mid q$ we define the  integer $\ell(d)$   by the conditions
$$
 2^{\ell(d)-1} U(d)  \le V(d) <  U(d) 2^{\ell(d)},
$$
and set
$$
V_i(d) = \min\left\{ 2^i  U(d) , V(d) \right\}, \quad i = 0, \ldots, \ell(d).
$$
Hence, we derive from~\eqref{eq: EaR} that
\begin{equation}
\label{eq: E1E2}
\EaR \le \frac{1}{q} \sum_{d \mid q} \sum_\pm\(|E_1^{\pm}(d) |+ \sum_{i=0}^{\ell(d)-1}|E_{2,i}^{\pm}(d)|\),
 \end{equation}
where
\begin{align*}
E_1^{\pm}(d)&=\sum_{a \in \cA}
\sum_{n \le  U(d)}  w_d(n) \cK_d(\pm n, a) \\
E_{2,i}^{\pm}(d)&=\sum_{a \in \cA}
\sum_{V_i(d) \le n < V_{i+1}(d) }  w_d(n)  \cK_d(\pm n, a),
\end{align*}
see also~\cite[Equation~(3.11)]{KeS20}.

In order to apply Theorem~\ref{thm:TypeI-complete-q}, we have to remove the coprime condition, since $\cA=(B,B+A]\cap \Z_q^*$ is not exactly an interval. To do so, we appeal to the M\"obius inversion, getting
\begin{align*}
E_1^{\pm}(d)&=\sum_{s\mid q}\mu(s)\sum_{\substack{B<a\leqslant B+A\\ s\mid a}}
\sum_{n \le  U(d)}  w_d(n) \cK_d(\pm n, a)\\
&=\sum_{s\mid q}\mu(s)\sum_{B/s<a\leqslant (B+A)/s}
\sum_{n \le  U(d)}  w_d(n) \cK_d(\pm n, sa).
\end{align*}
For $s>A$, there is at most one element in the $a$-sum, which contributes to $E_1^{\pm}(d)$
at most $XU(d)d^{-1/2+o(1)}\le d^{3/2}q^{o(1)}$. For $s\leqslant A$, the above transformations allow us to apply Theorem~\ref{thm:TypeI-complete-q} directly. We then arrive at
\begin{equation} \label{eq: E1 bound}
\begin{split}
E_1^{\pm}(d)
&\ll X^{\varepsilon}Ad^{3/2}\sum_{s\mid q}s^{-1}\Delta_1(d^2/X,A/s,d,\gcd(s,d))+d^{3/2}q^{o(1)}\\
&\ll (qX)^{\varepsilon}Ad^{3/2}\sum_{s|d}\frac{1}{s}\Delta_1(d^2/X,A/s,d,s)+d^{3/2}q^{o(1)} \\
&\ll (qX)^{\varepsilon}Ad^{3/2}\Delta_1(d^2/X,A,d,1)+d^{3/2}q^{o(1)}.
\end{split}
\end{equation}

Similar arguments also work for $E_{2,i}^{\pm}(d)$. In fact, we may obtain
\begin{equation} \label{eq: E2 mess}
\begin{split}
E_{2,i}^{\pm}(d)
&\ll X^{1/4}d^{1/2}V_i(d)^{-3/4}\\
&\qquad \qquad \(AV_i(d)d^{1/2}\Delta_1(V_i(d),A,d,1)+ V_i(d)d^{1/2+o(1)}\)\\
&\ll X^{1/4}dV_i(d)^{1/4}\(A\Delta_1(V_i(d),A,d,1)+1\)q^{o(1)}\\
&\ll AX^{1/4}dV_i(d)^{1/4}\Delta_1(V_i(d),A,d,1)+X^{1/2+\varepsilon}Y^{-1/2}d^{3/2}.
\end{split}
\end{equation}
Note that
\begin{align*}
V_i(d)^{1/4} \Delta_1(V_i(d) ,A,d,1) \le (2^i U(d))^{1/4} & \Delta_1(2^i U(d) ,A,d,1)\\
&\qquad + V(d)^{1/4} A^{-1/2},
\end{align*}
from which it follows that
\begin{equation} \label{eq: E2 fin}
\begin{split}
E_{2,i}^{\pm}(d)
\ll (qX)^\varepsilon\bigl( (2^i)^{1/4} Ad^{3/2}& \Delta_1(2^i d^2/X,A,d,1)\\
& \qquad \qquad \quad +(AX/Y)^{1/2}d^{3/2}\bigr)
\end{split}
\end{equation}
(in particular, we can drop the term $X^{1/2+\varepsilon}Y^{-1/2}d^{3/2}$ in~\eqref{eq: E2 mess}).

Clearly the bound~\eqref{eq: E2 fin} on  $E_{2,i}^{\pm}(d)$
dominates that of~\eqref{eq: E1 bound} on $E_1^{\pm}(d)$, which after the
substitution in~\eqref{eq: E1E2} and using the well-known bound on the divisor function
 (see, for example~\cite[Equation~(1.81)]{IK04}), yields
\begin{equation} \label{eq: E* bound}
\EaR  \ll  (qX)^{\varepsilon} \((qX)^{1/2}\Delta(X,A,q)+(qAX/Y)^{1/2}\)
\end{equation}
where
$$
\Delta(X,A,q)=\min\left\{X^{-1/4}+q^{-1/2},~A^{1/4}q^{-1/2}+Aq^{-1},~q^{-1/2}+Aq^{-3/4}\right\}.
$$
We note that $\Delta(X,A,q)$ is derived from $\Delta_1(q^2/X,A,q,1)$ where we drop the last term
$A^{-1/2}$ (which is already incorporated in~\eqref{eq: E2 fin} and thus in~\eqref{eq: E* bound})
and then we pull out the factor $X^{1/2} A^{-1}$.

Therefore, we now  infer from~\eqref{eq: EAB Eab} and~\eqref{eq: E1E2} that
$$
\EAR\ll  (qX)^{\varepsilon}\((qX)^{1/2}\Delta(X,A,q)+(qAX/Y)^{1/2}+A+AY/q\).
$$
Taking $Y =q(X/A)^{1/3}$ to balance the last two terms, we find
$$
\EAR\ll  (qX)^{\varepsilon}\((qX)^{1/2}\Delta(X,A,q)+A^{2/3}X^{1/3}\).
$$
Note that the above choice of $Y$ satisfies $Y\leqslant X/2$ since we assume $q^3<AX^2/8$.
This completes the proof of Theorem~\ref{thm:EAN}.

\section{Comments}
%\commI{I keep this comment for now}
%\commX{I revised this part a lot: some discussed are reworded; some arguments with percent signs have been removed; the discussion on $\RIJq$ with extra smooth weight functions are removed since the arguments are standards and should be familiar to readers}

Our results can be used in the same problems as the results of previous works~\cite{BFKMM17a, BFKMM17b, FKM14, KMS17, KMS20, Sh19, SZ16}. For example, using Theorem~\ref{thm:TypeI-complete-q} one can improve some results of~\cite{KeS20} on average values of the divisor function over some families of short arithmetic progressions.
% Perhaps they can also be used to improve the result of Blomer, Fouvry, Kowalski, Michel
% and  Mili{\'c}evi{\'c}~\cite[Theorem~1.8]{BFKMM17b} on sums of Kloosterman sums
% over prime values of parameters.

Theorems~\ref{thm:TypeII-incomplete-arbitraryset1} and~\ref{thm:TypeII-incomplete-arbitraryset2} can be further improved as long as $K$ is not too large. In fact, in the corresponding proofs, the variable $\mu$ is supported
on a set of cardinality at most $KV$. Therefore, the bound~\eqref{eq:Sigma3} can be improved as
$$
\Sigma_3\leqslant p\widetilde \Sigma_3 \ll p( K U^{r}V + U^{2r})
$$
for $KV\le p$.
This leads to the optimal choice
$$U = K^{2/(r+2)} \mand V = K^{(r-1)/(r+1)}.
$$
Hence, under the condition
$$
K \le p^{(r+1)/2r},
$$
instead of those in Theorems~\ref{thm:TypeII-incomplete-arbitraryset1} and~\ref{thm:TypeII-incomplete-arbitraryset2}  we now obtain better bounds (assuming $\|\balpha\|_\infty,\|\bgamma\|_\infty\leqslant1$)
$$
\WaMKp \ll  KMp^{o(1)} \(\frac{p}{K^{2r/(r+1)}  M} \)^{1/4r}
%\(\frac{1}{M}  + \frac{p}{K^{2r/(r+1)}  M} \)^{1/4r}
$$
and
$$
\WaMKp\ll    KMp^{o(1)}
 \(\frac{1}{M^{2r}}+ \frac{1}{K}    + \frac{p}{K^{2r/(r+1)}  M^{10/13}} \)^{1/4r},
$$
respectively.    In turn, this leads to corresponding modifications of Theorems~\ref{thm:TypeI-complete-arbitraryset1} and~\ref{thm:TypeI-complete-arbitraryset2}.

As in~\cite{Sh19}, we can also apply our results to the double sums
$$
\RIJq = \sum_{m\in \cI} \sum_{n \in \cJ}\cR_q(m,n,1),
$$
where $\cR_q(m,n,1)$ is the double Kloosterman sum given by
$$
\cR_q(m,n,\ell )=\mathop{\sum\sum}_{x,y\in\Z_q^*} \eq(mx+ny+\ell \xbar\ybar).
$$

The elementary arguments turn out to be very powerful in the studies on bilinear forms with Kloosterman sums, and the proof relies heavily on the exact shapes of such sums. The methods from $\ell$-adic cohomology employed in~\cite{FKM14,FKM15,KMS17,KMS20} are quite deep and applicable to a large family of functions rather than Kloosterman sums only. It should be very meaningful and exhilarating to investigate if the above two approaches can be combined, and  lead to stronger results than either of them separately.

\appendix

\section{Evaluations of Kloosterman sums}\label{sec:Kloostermanappendix}
\subsection{Quadratic Gauss sums}
Before presenting results related to the evaluation of Kloosterman sums, we require some facts about quadratic Gauss sums
$$
G(m,n;q)
=\sum_{a\in\Z_q}\eq(ma^2+na).
$$
The following statements are well-known (see, for example,~\cite[\S~6]{Es62}).
\begin{lemma}\label{lem:Gausssum} Suppose $d=\gcd(m,q)$.
\begin{enumerate}[$(i)$]
\item If $d=1$, then $|G(m,n;q)|\leqslant2\sqrt{q}$.
\item $G(m,n;q)$ vanishes unless $d\mid n$, in which case we have
$$
G(m,n;q)=d G\(\frac{m}{d},\frac{n}{d};\frac{q}{d}\).
$$
\end{enumerate}
\end{lemma}

In fact, we   encounter the following modified quadratic Gauss sums
\begin{equation}\label{eq:Gausssum*}
G^*(m,n;q) =\sum_{a\in\Z_q^*}\eq(ma^2+na)
\end{equation}
in subsequent evaluations of Kloosterman sums. To associate the above two sums, we may appeal to the
M\"obius formula, getting
 \begin{equation}
\label{eq:G-G*}
\begin{split}
G^*(m,n;q)
&=\sum_{d\mid q}\mu(d)\sum_{a\in\Z_{q/d}}\eq(m(ad)^2+nad) \\
&=\sum_{d\mid q}\mu(d)G(md,n;q/d).
\end{split}
\end{equation}
This, together with Lemma~\ref{lem:Gausssum} allows us to derive the following inequality.
\begin{lemma}\label{lem:Gausssum*-upperbound}
Let $q$ be a positive integer. For any $m,n\in\Z$, we have
$$
G^*(m,n;q) \ll q^{1/2+o(1)}\gcd(m,n,q)^{1/2}.
$$
\end{lemma}

\begin{proof}
Note that $G(md,n;q/d)$ in~\eqref{eq:G-G*} vanishes unless $\gcd(md,q/d)\mid n$ by Lemma~\ref{lem:Gausssum}, in which case we have
$$
G(md,n;q/d)\ll \gcd(md,q/d)^{1/2}(q/d)^{1/2}\leqslant q^{1/2}\gcd(m,q/d)^{1/2}.
$$
We then infer
\begin{align*}
G^*(m,n;q)
&\ll q^{1/2}\sum_{d\mid q,~\gcd(m,q/d)\mid n}\gcd(m,q/d)^{1/2}\\
&=q^{1/2}\sum_{d\mid q,~\gcd(m,q/d)\mid n}\gcd(m,q/d,n)^{1/2},
\end{align*}
which readily gives Lemma~\ref{lem:Gausssum*-upperbound}.
\end{proof}

\subsection{Vanishing of Kloosterman sums}
We now turn to Kloosterman sums, and the following two lemmas characterize when such sums can vanish.

\begin{lemma}\label{lem:Kloosterman=0-initial}
Suppose $q=p^j$ with a prime $p$ and an integer  $j\geqslant2$. For any $m\in\Z$ with $p\mid m$, we have
$$
\cK_q(m,1)=0.
$$
\end{lemma}

\begin{proof}
Note that $x+yp^{j-1}$ presents all elements of $\Z_q^*$ exactly  as long as $x$ runs over $\Z_{p^{j-1}}^*$ 
and $y$ runs over $\Z_p$, respectively.
Therefore,
\begin{align*}
\cK_q(m,1)
&=\sum_{x\in\Z_{p^{j-1}}^*}\sum_{y\in\Z_p}\e_{p^j}(m(x+yp^{j-1})^{-1}+x+yp^{j-1})\\
&=\sum_{x\in\Z_{p^{j-1}}^*}\sum_{y\in\Z_p}\e_{p^j}(mx^{-1}+x+yp^{j-1})
\end{align*}
since $p\mid m$. The inner sum over $y$ now vanishes due to the orthogonality of additive characters. This completes the proof.
\end{proof}

\begin{lemma}\label{lem:Kloosterman=0}
Suppose $q=p^j$ with a prime $p$ and an integer  $j\geqslant2$. For any $m,n\in\Z$ with
$$
q\nmid m,\quad q\nmid n,\quad \gcd(m,q)\neq \gcd(n,q),
$$
we have
$$
\cK_q(m,n)=0.
$$
\end{lemma}

\begin{proof}
Without loss of generality, we assume $m=p^\alpha m_1$, $n=p^\beta n_1$ with $\alpha<\beta<j$ and $p\nmid m_1n_1$.
From~\eqref{eq:Selberg-Kuznetsov} it follows that
$$
\cK_q(m,n)
=\sum_{0\leqslant s\leqslant \alpha}p^s\cK_{p^{j-s}}(m_1n_1p^{\alpha+\beta-2s},1).
$$
Note that $j-\alpha\geqslant2$ and $\alpha+\beta-2s\geqslant1$ for any $0\leqslant s\leqslant \alpha$. We are now in a position to apply Lemma~\ref{lem:Kloosterman=0-initial} to each Kloosterman sum on the right hand side, and this proves the lemma immediately.
\end{proof}

\subsection{Exact expressions of Kloosterman sums}

The following lemma allows one to extract $\gcd(m,n,q)$ from $\cK_q(m,n)$ if $q\nmid \gcd(m,n)$.

\begin{lemma}\label{lem:Kloosterman-pulloutd}
Suppose $q=p^j$ with a prime $p$ and an integer  $j\geqslant2$ and $m,n\in\Z$. Let $d=\gcd(m,n,q)$. If $d\neq q$, then we have
$$
\cK_q(m,n)=d\cK_{q^*}(m^*,n^*),
$$
where $m^*=m/d$,  $n^*=n/d$ and $q^*=q/d$.
\end{lemma}

\begin{proof}
From~\eqref{eq:Selberg-Kuznetsov} it follows that
$$
\cK_q(m,n)
=\sum_{r\mid d}r\cK_{q/r}(mnr^{-2},1).
$$
Since $d\neq q$, we find $p^2\mid q/r$ and $p\mid mnr^{-2}$ unless $r=d$. By Lemma~\ref{lem:Kloosterman=0-initial}, only the contribution from $r=d$ survives, and this proves Lemma~\ref{lem:Kloosterman-pulloutd} immediately.
\end{proof}

For the explicit evaluations of  Kloosterman sums $\cK_q(m,n)$ with $q=p^j$ and $j\ge2$, it suffices to consider the case $\gcd(m,n,q)=1$ in view of Lemma~\ref{lem:Kloosterman-pulloutd}, which condition is equivalent to $p\nmid\gcd (m,n)$. Furthermore, the vanishing property in Lemma~\ref{lem:Kloosterman=0} leads us to consider the case $\gcd(mn,q)=1$.

We now formulate an exact expression of $\cK_q(m,n)$ modulo an odd prime power $q$ with $\gcd(mn,q)=1$, and the case modulo ${2^j}$ can be derived in a similar manner.
One may refer to Iwaniec~\cite[Proposition~4.3]{Iwa97} for original resources.

\begin{lemma}\label{lem:Kloosterman-exact}
Suppose $q=p^j$ with a prime $p$ and an integer  $j\geqslant2$ and $\gcd(q,2mn)=1$. We have $\cK_q(m,n)=0$ unless $m\equiv l^2n \bmod q$ for some $l\in\Z$, in which case there is
$$
\cK_q(m,n)=\(\frac{ln}{q}\)q^{1/2} \Re\{\ve_q \e_q(2ln)\},
$$
where $(\frac{\cdot}{q})$ denotes the Legendre-Jacobi symbol $\bmod q$, and $\ve_q=1$ or $i$ according to $q\equiv 1$ or $-1 \bmod 4$.
\end{lemma}

\subsection{Estimates for $T(x,y,z;q)$}
%%commB Reworded
Our proof of Lemma~\ref{lem:T-upperbound-q} requires first considering some special cases.

\begin{lemma}\label{lem:T-upperbound-(xy,q)=1}
Let $q=p^j$ with a prime $p$ and an integer  $j\geqslant2$. For $x,y,z\in \Z$ with $\gcd(xy,q)=1$, we have
$$
T(x,y,z;q)\ll \gcd(x-y,z,q)^{1/2}q^{1/2+o(1)}.
$$
\end{lemma}

\begin{proof}
For purely technical reasons, we only give the details for $p>2$, %commB and the -> the
the $p=2$ case can be treated in a similar way.
From~\eqref{eq:Kloosterman-transition} and Lemmas~\ref{lem:Kloosterman=0-initial} and~\ref{lem:Kloosterman-exact}, we find
$$
\cK_q(x,t)\cK_q(y,t)=0
$$
unless $t\in \Z_q^*$ and that both of $tx^{-1}$ and $ty^{-1}$ are quadratic residues modulo $q$, in which case we have
$$
t\equiv l^2x\equiv l^2c^2y \bmod q
$$
with $c^2y\equiv x\bmod q$.
Hence
\begin{align*}
T(x,y,z;q)
&=\frac{1}{q}\sum_{t\in\Z_q^*}\cK_q(x,t)\cK_q(y,t)\eq(-zt)\\
&=\frac{1}{2q}\sum_{l\in\Z_q^*}\cK_q(x,l^2x)\cK_q(c^{-2}x,l^2x)\eq(-xzl^2).
\end{align*}
From Lemma~\ref{lem:Kloosterman-exact} it follows that
$$
T(x,y,z;q)
 =\frac{1}{2}\(\frac{xyc}{q}\)\sum_{l\in\Z_q^*}\Re\{\ve_q \e_q(2lx)\} \Re\{\ve_q \e_q(2lc^{-1}x)\} \eq(-zl^2x).
$$
Expanding the real parts, $T(x,y,z;q)$   bear four quadratic exponential sums of shapes $G^*(\cdot,\cdot;q)$ as defined by~\eqref{eq:Gausssum*}. 
From Lemma~\ref{lem:Gausssum*-upperbound}, it follows that
$$
T(x,y,z;q) \ll q^{1/2+o(1)}\sum_\pm \gcd(zx,(1\pm c^{-1})x,q)^{1/2}.
$$
Since $c^2y\equiv x\bmod q$ and $\gcd(xy,q)=1$, we find
$$
\gcd((1\pm c^{-1})x,q)=\gcd(x\pm cy,q),
$$
which divide
$$
\gcd(x^2-c^2y^2,q)=\gcd(x^2-xy,q)=\gcd(x-y,q).
$$
This completes the proof.
\end{proof}

\begin{lemma}\label{lem:T-upperbound-p}
Let $p$ be a prime. For any $x,y,z\in\Z$, we have
$$
T(x,y,z;p)
\ll \gcd(x,y,p)^{1/2}\gcd\(x-y,z,\frac{p}{(x,y,p)}\)^{1/2}p^{1/2}.
$$
\end{lemma}

\begin{proof}
If $p\mid z$, we find
$$
T(x,y,0;p)=\frac{1}{p}\sum_{t\in\Z_p}\cK_q(x,t)\cK_q(y,t)=c_p(x-y)
$$
is a Ramanujan sum.
The desired estimate then follows from the elementary inequality~\eqref{eq:Ramanujansum-upperbound}.
For $p\nmid z$, the orthogonality of additive characters yields
$$
T(x,y,z;p)
 =\sum_{a\in\Z_p^*\setminus\{z^{-1}\}}\ep(ax+a(az-1)^{-1}y).
$$
With the change of variable $a\rightarrow z^{-1}(a+1)$, we derive that
\begin{align*}
T(x,y,z;p)
&=\sum_{a\in\Z_p^*\setminus\{-1\}}\ep({x{z}^{-1}a+y{z}^{-1}a^{-1}})\ep((x+y)z^{-1})\\
&=\cK_p(x{z}^{-1},y{z}^{-1})\ep((x+y)z^{-1})-1.
\end{align*}
Now the results  follows directly from the Weil bound~\eqref{eq:Weil}.
\end{proof}

\subsection{Proof of Lemma~\ref{lem:T-upperbound-q}}
We are now ready to prove Lemma~\ref{lem:T-upperbound-q} for a general modulus $q$. In view of the twisted multiplicativity~\eqref{eq:T-multiplicativity} and Lemma~\ref{lem:T-upperbound-p}, it suffices to consider the situation of prime power moduli $q=p^j$ with $j\geqslant2$,
and we split the performance to two cases according to the divisibility of $x,y$ by $q$.

\underline{{\bf Case I}: $q\mid x$ or $q\mid y$.} Without loss of generality, we may assume $q\mid x$. From~\eqref{eq:Ramanujansum-upperbound}, it follows that
\begin{align*}
|T(x,y,z;q)|
&\leqslant \frac{1}{q}\sum_{t\in\Z_q}\gcd(t,q)|\cK_q(y,t)|\\
&\leqslant \gcd(y,q)+q^{-1/2+o(1)}\gcd(y,q)^{1/2}\sum_{1\leqslant t<q}\gcd(t,q)\\
&\leqslant q^{1/2+o(1)}\gcd(y,q)^{1/2}
\end{align*}
as desired.

\underline{{\bf Case II}: $q\nmid x$ and $q\nmid y$.}
We first extract the zero-th frequency, and the bound~\eqref{eq:Ramanujansum-upperbound} for Ramanujan sums gives
\begin{equation}\label{eq:T-T1}
|T(x,y,z;q)| \leqslant |T_1(x,y,z;q)|+\frac{\gcd(x,q)\gcd(y,q)}{q}
\end{equation}
with
$$
T_1(x,y,z;q)
=\frac{1}{q}\sum_{1\leqslant t<q}\cK_q(x,t)\cK_q(y,t)\eq(-zt).
$$

For each $1\leqslant t<q$, Lemma~\ref{lem:Kloosterman=0} yields that $\cK_q(x,t)\cK_q(y,t)$ vanishes unless
$$
\gcd(x,q)=\gcd(y,q)=\gcd(t,q)=p^k
$$
for some $k=k(t)<j$,
in which case Lemma~\ref{lem:Kloosterman-pulloutd} guarantees
$$
\cK_q(x,t)=p^k\cK_{q_0}(x_0,t_0), \qquad  \cK_q(y,t)=p^k\cK_{q_0}(y_0,t_0)
$$
with $x_0=x/p^k, y_0=y/p^k, t_0=t/p^k, q_0=q/p^k$. Hence we have $\gcd(x_0y_0,q_0)=1$ and
\begin{equation}\label{eq:T1-T1x0y0q0}
T_1(x,y,z;q)=p^kT_1(x_0,y_0,z;q_0).
\end{equation}
Note that
$$
T(x_0,y_0,z;q_0)-T_1(x_0,y_0,z;q_0)\ll \frac{1}{q_0}
$$
in view of~\eqref{eq:Ramanujansum-upperbound},
and Lemma~\ref{lem:T-upperbound-(xy,q)=1} yields
$$
T_1(x_0,y_0,z;q_0)\ll \frac{1}{q_0}+\gcd(x_0-y_0,z,q_0)^{1/2}q_0^{1/2+o(1)},
$$
from which and~\eqref{eq:T-T1}, \eqref{eq:T1-T1x0y0q0}, Lemma~\ref{lem:T-upperbound-q} follows immediately.

\section*{Acknowledgements}

The authors would like to thank Christian Bagshaw for the careful reading of the manuscript
and pointing out some impecisions in the initial version.

During the preparation of this work, B.~Kerr was supported by the Max Planck Institute for Mathematics,
I.~E.~Shparlinski was supported in part by the ARC  (DP170100786), X.~Wu was supported in part by the  NSFC (No.~12271135), and
P.~Xi was supported in part by the NSFC (No.~12025106, No.~11971370) and by The Young Talent Support Plan in Xi'an Jiaotong University.


\begin{thebibliography}{9999}

\bibitem{ACZ96}
A. Ayyad, T. Cochrane and Z. Zheng,
The congruence $x_1x_2\equiv x_3x_4\bmod p$, the equation $x_1x_2=x_3x_4$ and mean values of character sums,
\emph{J. Number Theory} \textbf{59} (1996), 398--413.

%%\bibitem{BS21}
%%N. Bag and  I. E. Shparlinski,
%%Bounds of some double exponential sums,
%%\emph{J. Number Theory} \textbf{219} (2021), 228--236.

\bibitem{BS23}
N. Bag and  I. E. Shparlinski,
Bounds on bilinear sums of Kloosterman sums,
\emph{J. Number Theory} \textbf{242} (2023) 102--111.


\bibitem{BaSh20} W. Banks and I. E. Shparlinski, Congruences with intervals and arbitrary sets,
\emph{Archiv  Math.} \textbf{114} (2020),  527--539.

\bibitem{BHBS05} W. D. Banks,  R. Heath-Brown and  I. E.~Shparlinski,
On the average value of divisor sums in arithmetic progressions,
\emph{Int. Math. Res. Not.} \textbf{2005} (2005), 1--25.

\bibitem{Blom08}  V. Blomer, The average value of divisor sums in
arithmetic progressions,
\emph{Quart. J. Math.} \textbf{59} (2008), 275--286.

\bibitem{BFKMM17a}
V. Blomer, \'E. Fouvry, E. Kowalski, Ph. Michel and D. Mili{\'c}evi{\'c},
 On moments of twisted $L$-functions,
\emph{Amer. J.  Math.} \textbf{139} (2017), 707--768.


\bibitem{BFKMM17b} V. Blomer, \'E. Fouvry, E. Kowalski, Ph. Michel and D. Mili{\'c}evi{\'c},
Some applications of smooth bilinear forms with Kloosterman sums,
\emph{Trudy Matem.  Inst. Steklov}  \textbf{296} (2017), 24--35;  translation in
\emph{Proc. Steklov Math. Inst.} \textbf{296} (2017), 18--29.


\bibitem{BFKMMS23}
V. Blomer, {\'E}. Fouvry, E. Kowalski, Ph. Michel, D. Mili{\'c}evi{\'c} and W. Sawin,
The second moment theory of families of $L$-functions,
 \emph{Memoirs Amer. Math. Soc.} \textbf{282} (2023), Number 1394.



\bibitem{Bo05}
J. Bourgain,
More on the sum-product phenomenon in prime fields and its applications,
 \emph{Int. J. Number Theory} \textbf{1} (2005), 1--32.



\bibitem{CFK+05}
J. B. Conrey, D. W. Farmer, J. P. Keating, M. O. Rubinstein and N. C. Snaith,
Integral moments of $L$-functions,
\emph{Proc. London Math. Soc.} \textbf{91} (2005), 33--104.

\bibitem{CG11}
J. Cilleruelo and M. Z. Garaev,
Concentration of points on two and three dimensional modular hyperbolas and applications',
\emph{GAFA} \textbf{21} (2011), 892--904.

\bibitem{CS10}
T. Cochrane and S. Shi,
The congruence $x_1x_2 \equiv x_3x_4 \pmod m$ and mean values of character sums,
\emph{J. Number Theory} \textbf{130} (2010), 767--785.

\bibitem{De80}
P. Deligne,
La conjecture de Weil, II,
\emph{Publ. Math. IH\'ES} \textbf{52} (1980), 137--252.

\bibitem{DKSZ20}
A. Dunn, B. Kerr, I. E. Shparlinski  and A. Zaharescu,
Bilinear forms in Weyl sums for modular square roots and applications,
\emph{Adv. Math.}  \textbf{375} (2020),  Art.~107369.


\bibitem{DZ19}
A. Dunn and A. Zaharescu,
The twisted second moment of modular half integral weight $L$-functions,
\emph{Preprint}, 2019 (available from \url{http://arxiv.org/abs/1903.03416}.


\bibitem{Es62}
T. Estermann,
A new application of the Hardy--Littlewood--Kloosterman method,
\emph{Proc. London Math. Soc.} \textbf{12} (1962), 425--444.


\bibitem{FKM14}
\'E. Fouvry,  E. Kowalski and Ph.  Michel, Algebraic trace functions over the primes,
\emph{Duke Math. J.} {\bf 163} (2014),  1683--1736.


\bibitem{FKM15}
\'E. Fouvry, E. Kowalski and Ph. Michel,
On the exponent of distribution of the ternary divisor function,
\emph{Mathematika} \textbf{61} (2015), 121--144.


\bibitem{FM98}
\'E. Fouvry and Ph. Michel,
Sur certaines sommes d'exponentielles sur les nombres premiers,
\emph{Ann. Sci. \'Ecole Norm. Sup.} \textbf{31} (1998), 93--130.

\bibitem{FI85}
J. B. Friedlander and H. Iwaniec,
Incomplete Kloosterman sums and a divisor problem (with an appendix by B.
J. Birch and E. Bombieri),
\emph{Ann. of Math.} \textbf{121} (1985), 319--350.


\bibitem{FI93}
J. B. Friedlander and H. Iwaniec,
Estimates for character sums,
\emph{Proc. Amer. Math. Soc.} \textbf{119} (1993), 365--372.



\bibitem{HB78}
D. R. Heath-Brown,
Almost-primes in arithmetic progressions and short intervals,
\emph{Math. Proc. Cambridge Philos. Soc.} \textbf{83} (1978), 357--375.


\bibitem{HB86}
D. R. Heath-Brown,
The divisor function $d_3(n)$ in arithmetic progressions,
\emph{Acta Arith.} \textbf{47} (1986), 29--56.

\bibitem{Ho57}
C. Hooley,
An asymptotic formula in the theory of numbers,
\emph{Proc. London Math. Soc.} \textbf{7} (1957), 396--413.



\bibitem{Iwa97}
H. Iwaniec,
Topics in Classical Automorphic Forms,
Graduate Studies in Mathematics, {\bf 17}, Amer. Math. Soc., Providence, RI, 1997.

\bibitem{IK04}
H.  Iwaniec and E.  Kowalski,
Analytic Number Theory, Amer. Math.  Soc., Providence, RI, 2004.

\bibitem{Kar}
A. A. Karastuba,
Sums of characters over prime numbers,
\emph{Izv. Akad. Nauk SSSR Ser. Mat.} \textbf{34} (1970), 299--321.

\bibitem{Ke17}
B. Kerr,
On the congruence $x_1x_2\equiv x_3x_4 \pmod q$,
\emph{J. Number Theory} \textbf{180} (2017), 154--168.

\bibitem{KSSZ21}
B. Kerr, I. D. Shkredov, I. E. Shparlinski and A. Zaharescu,
Energy bounds for modular roots and their applications,
\emph{Preprint}, 2021 (available from \url{http://arxiv.org/abs/2103.09405}.


\bibitem{KeS20}
B. Kerr and I. E. Shparlinski,
Bilinear sums of Kloosterman sums, multiplicative congruences and average values of the divisor function over families of arithmetic progressions,
\emph{Res. Number Theory} \textbf{6}  (2020), Art.~16.

 \bibitem{KiSa}
H. Kim and P. Sarnak, Appendix to  ``H. Kim, Functoriality for the exterior square of ${\mathrm{GL}}(4)$ and the symmetric fourth of  ${\mathrm{GL}}(2)$",  \emph{J. Amer. Math. Soc.} \textbf{16} (2003),   175--181.

\bibitem{KoS20}
M. Korolev and I. E. Shparlinski,
Sums of algebraic trace functions twisted by arithmetic functions,
\emph{Pacif. J. Math.}  \textbf{304} (2020), 505--522.

\bibitem{KMS17}
E. Kowalski, Ph. Michel and W. Sawin,
Bilinear forms with Kloosterman sums and applications,
\emph{Ann. of Math.} \textbf{186} (2017), 413--500.

\bibitem{KMS20}
E. Kowalski, Ph. Michel and W. Sawin,
Stratification and averaging for exponential sums: Bilinear forms with generalized Kloosterman sums,
\emph{Ann. Scuola Normale Pisa} \textbf{21} (2020), 1453--1530.
%
%\bibitem{L-OST} R. J. Lemke Oliver,  S.  Shrestha, and F.  Thorne,
%`Asymptotic identities for additive convolutions of sums of divisors',
% Preprint, 2021, \url{https://arxiv.org/abs/2007.09275}.

%% \bibitem{L-OSo} R. J. Lemke Oliver and K. Soundararajan,
%%`The distribution of consecutive prime biases and sums of sawtooth random variables',
%%\emph{Math. Proc. Camb. Phil.},  {\bf 168},  (2020), 149--169.
%%
%%\bibitem{LN} R. Lidl and H. Niederreiter,
%%\emph{Finite fields}, Cambridge Univ. Press, Cambridge, 1997.

%\bibitem{Li61}
%Yu. V. Linnik,
%All large numbers are sums of a prime and two squares (A problem of Hardy and Littlewood). II,
%\emph{Mat. Sb. (N.S.)} \textbf{53(95)} (1961), 3--38.

\bibitem{LSZ18a}
K. Liu,  I. E. Shparlinski and T. P. Zhang,
Divisor problem in arithmetic progressions modulo a prime power,
\emph{Adv. Math.} \textbf{325} (2018), 459--481.

\bibitem{LSZ18b}
K. Liu,  I. E. Shparlinski and T. P. Zhang,
Bilinear forms with exponential sums with binomials,
\emph{J. Number Theory} \textbf{188} (2018), 172--185.

\bibitem{LSZ19}
K. Liu,  I. E. Shparlinski and T. P. Zhang,
Cancellations between  Kloosterman sums  modulo a prime power with prime arguments,
\emph{Mathematika} \textbf{65} (2019), 475--487.

%%\bibitem{MS19}
%%S.  Macourt and  I. E. Shparlinski,
%%Double sums of Kloosterman sums in finite fields,
%%\emph{Finite Fields  Appl.} \textbf{60} (2019), no.101575.

\bibitem{MPSS20}
S. Macourt, G. Petridis, I. D. Shkredov and I. E. Shparlinski,
Bounds of trilinear and trinomial exponential sums,
\emph{SIAM J. Discr. Math.} \textbf{34} (2020), 2124--2136.



%\bibitem{Nu17}
%R. M. Nunes,
%On the least squarefree number in an arithmetic progression,
%\emph{Mathematika} \textbf{63} (2017),  483--498.

\bibitem{Nun16}  R. M. Nunes,
 `Squarefree numbers in large arithmetic progressions',
\emph{Preprint}, 2016 (available from \url{http://arxiv.org/abs/1602.00311}).




\bibitem{RSS20}
M. Rudnev, I. D. Shkredov and S. Stevens,
On the energy variant of the sum-product conjecture,
\emph{Rev. Mat. Iberoam.} \textbf{36} (2020), 207--232.

\bibitem{RW22}
M. Rudnev and J. Wheeler,
On incidence bounds with M{\"o}bius hyperbolae in positive characteristic,
\emph{Finite Fields and  Appl.} \textbf{78} (2022), Art.~101978.

\bibitem{Sh21}
I. D. Shkredov, 
Modular hyperbolas and bilinear forms of Kloosterman sums,
\emph{J. Number Theory} \textbf{220} (2021), 182--211.


\bibitem{SSZ20}
I. D. Shkredov, I. E. Shparlinski  and A. Zaharescu,
On the distribution of modular square roots of primes,
\emph{Preprint}, 2020 (available from \url{http://arxiv.org/abs/2009.03460}.


\bibitem{SSZ22}
I. D. Shkredov, I. E. Shparlinski and A. Zaharescu,
Bilinear forms with modular square roots and averages of twisted second moments of half integral weight $L$-functions,
\emph{Int. Math. Res. Not.} \textbf{2022} (2022), 17431--17474.


\bibitem{Sh19}
I. E. Shparlinski,
On sums of Kloosterman and Gauss sums,
\emph{Trans. Amer. Math. Soc.} \textbf{371} (2019), 8679--8697.

\bibitem{SZ16}
I. E. Shparlinski and T. P. Zhang,
Cancellations amongst Kloosterman sums,
\emph{Acta Arith.} \textbf{176} (2016), 201--210.


\bibitem{WX21}
J. Wu and P. Xi,
Arithmetic exponent pairs for algebraic trace functions and applications, with an appendix by W.~Sawin,
\emph{Algebra Number Theory} \textbf{15} (2021), 2123--2172.

\bibitem{Wu22}
X. Wu,
The fourth moment of Dirichlet $L$-functions at the central value,
\emph{Math. Ann.} (2022), \url{https://doi.org/10.1007/s00208-022-02483-9}.


\bibitem{Xi17}
P. Xi,
Large sieve inequalities for algebraic trace functions, with an appendix by \'E. Fouvry, E. Kowalski and Ph. Michel,
\emph{Int. Math. Res. Not.} \textbf{2017} (2017), 4840--4881.


\bibitem{Xi18}
P. Xi,
Ternary divisor functions in arithmetic progressions to smooth moduli,
\emph{Mathematika} \textbf{64} (2018), 701--729.


\bibitem{Yo11}
M. Young,
The fourth moment of Dirichlet $L$-functions,
\emph{Ann. of Math.} \textbf{173} (2011), 1--50.

\bibitem{Zac19}
R. Zacharias,
Mollification of the fourth moment of Dirichlet $L$-functions,
\emph{Acta Arith.} \textbf{191} (2019), 201--257.

\bibitem{Zh22}
L. Zhao,
On products of primes and almost primes in arithmetic progressions,
\emph{Acta Arith.} \textbf{204} (2022), 253--267.

\end{thebibliography}
\end{document}